\documentclass [11pt] {article}
\usepackage{amsfonts}
\usepackage{color}
\usepackage{amsthm}
\usepackage{amsmath}
\usepackage{graphicx}
\usepackage{url}
\usepackage{txfonts}
\usepackage{epstopdf}
\usepackage{multirow}
\usepackage{booktabs}
\usepackage{caption}
\usepackage{float}
\renewcommand{\baselinestretch}{1.2}
\newtheorem{theorem}{Theorem}[section]

\newtheorem{lemma}[theorem]{Lemma}

\newtheorem{algorithm}[theorem]{Algorithm}
\renewcommand{\thefootnote}

\newcommand{\be}{\begin{equation}}
\newcommand{\ee}{\end{equation}}

\begin{document}

\renewcommand{\baselinestretch}{1.2}

\title {Finding a Nonnegative Solution to an M-Tensor Equation\thanks{Supported by the NSF of China grant 11771157.}}

\author{Dong-Hui Li\footnote{School of Mathematical Sciences, South China Normal University, Guangzhou, 510631,
China. E-mail: lidonghui@m.scnu.edu.cn} \and Hong-Bo Guan\footnotemark[2] \ \footnote{School of Mathematical Sciences and Energy Engineering,
Hunan Institute of Technology, HengYang, 421002, China.E-mail: hongbo\_guan@m.scnu.edu.cn}\and  Xiao-Zhou Wang\thanks{Department of
Applied Mathematics, Hong Kong Polytechnic University, Hong Hom, Hong Kong.E-mail:xzhou.wang@connect.polyu.hk}
}

\maketitle

\begin{abstract}
We are concerned with the tensor equation with an M-tensor, which we call the M-tensor equation.
We first derive a necessary and sufficient condition for an M-tensor equation to have
nonnegative solutions.  We then develop a monotone iterative
method to find a nonnegative solution to an M-tensor equation.
The method can be regarded as an approximation to Newton's method for
solving the equation. At each iteration, we solve a system of linear
equations. An advantage of the proposed method is that the coefficient
matrices of the linear systems are independent of the iteration.  We show that if the initial point  is appropriately
chosen, then the sequence of iterates generated by the method converges
to a nonnegative solution of the M-tensor equation monotonically and linearly.
At last, we do numerical experiments to test the proposed methods.
The results show the efficiency of the proposed methods.
\end{abstract}

{\bf Keywords} M-tensor equation, iterative method, monotone convergence.

\section{Introduction}

Consider the tensor equation
 \begin{equation}\label{multi-linear equation}
{\cal A}x^{m-1} = b,
\end{equation}
where $x,b\in R^n$ and ${\cal A}$ is an ${m}$th-order ${n}$-dimensional tensor that takes the form
\[
{\cal A} = (a_{i_1i_2\ldots i_m}), \quad  a_{i_1i_2\ldots i_m}\in R, \quad 1\le i_1, i_2, \cdots, i_m \le n,
\]
and ${{\cal A}x^{m-1}}\in R^n$  with elements
 \[
({\cal A}x^{m-1})_i = \sum_{i_2, \ldots, i_m=1}^na_{ii_2\ldots i_m}x_{i_2}\cdots x_{i_m},\quad i=1,2,\ldots,n.
\]
Tensor equation is also called multilinear equation. It appears in many practical fields including data mining and numerical partial differential
equations \cite {4,Ding-Wei-16,8,Fan-Zhang-Chu-Wei-17,Kressner-Tobler-10, Li-Xie-Xu-17, Li-Ng-15, Xie-Jin-Wei-2017}.

We denote the set of all ${m}$th-order ${n}$-dimensional tensors by ${\cal T}(m,n)$.
A tensor ${\cal A} = (a_{i_1i_2\ldots i_m})\in {\cal T}(m,n)$ is called symmetric tensor if its elements $a_{i_1i_2\ldots i_m}$
are invariant under any permutation of their indices $(i_1, i_2, \cdots, i_m)$.
The set of all ${m}$th-order ${n}$-dimensional symmetric tensors is denoted by ${\cal ST}(m,n)$.
In the case ${\cal A}\in {\cal ST}(m,n)$, it holds that $\nabla ({\cal A}x^m)=m{\cal A}x^{m-1}$.
Consequently, the tensor equation (\ref{multi-linear equation}) is just the first order necessary condition of the
optimization problem
\[
\min f(x)=\frac 1m {\cal A}x^m-b^Tx.
\]

Tensor equation is a special system of nonlinear equations.
It may be solved by existing numerical methods such as Newton and quasi-Newton methods
for solving nonlinear equations. However, in some cases, Newton's method may fail
or be only linearly convergent
if
the Jacobian matrix of $F$ is singular.

\noindent
{\bf Example 1.1.} Consider  the quadratic polynomial equation
\begin{equation}\label{ex-1}
F(x)= \left ( \begin{array}{l}
x_1^2-1\vspace{1mm} \\
x_1^2+x_2^2-1 \\
-x_1^2+x_2^2+x_3^2+1
\end{array}\right )=0.
\end{equation}
It corresponds to the tensor equation (\ref{multi-linear equation}) with $b=(1,1,-1)^T$ and ${\cal A}=(a_{ijk})$
whose elements are
\[
a_{111}=a_{222}=a_{333}=1,\; a_{211}=a_{322}=1,\; a_{311}=-1.
\]
and all other elements are zeros. The equation have two solutions
$x^*=(1, 0,0)^T$ and $ \tilde x^*=(-1, 0,0)^T$.

The Jacobian of $F$ at $x$ is
\[
F'(x)= 2 \left ( \begin{array}{ccc}
x_1 & 0   & 0 \vspace{1mm} \\
x_1 & x_2 & 0 \vspace{1mm} \\
-x_1 & x_2 & x_3
\end{array}\right ) = 2\left ( \begin{array}{ccc}
1 & 0 & 0 \vspace{1mm} \\
1 & 1 & 0 \vspace{1mm} \\
-1 & 1 & 1
\end{array}\right ) \left ( \begin{array}{ccc}
x_1 &  & \vspace{1mm} \\
  & x_2 & \vspace{1mm} \\
    & & x_3
\end{array}\right ).
\]
It is singular if $x_1=0$, $x_2=0$ or $x_3=0$.
In particular, $F'(x^*)$ and $F'(\tilde x^*)$ are singular.

Let $x$ be the current iterate that has no zero elements. The next iterate
generated by Newton's method $x^+$ is determined by
\begin{eqnarray*}
\lefteqn {x^+ = x  - F'(x)^{-1}F(x) }\\
    &=& x-\frac 12 \left ( \begin{array}{ccc}
x_1^{-1} &  & \vspace{1mm} \\
  & x_2^{-1} & \vspace{1mm} \\
  & & x_3^{-1}
\end{array}\right )\left ( \begin{array}{ccc}
1 & 0 & 0 \vspace{1mm} \\
-1 & 1 & 0 \vspace{1mm} \\
2 & -1 & 1
\end{array}\right ) \Big [\left ( \begin{array}{ccc}
1 & 0 & 0 \vspace{1mm} \\
1 & 1 & 0 \vspace{1mm} \\
-1& 1 & 1
\end{array}\right )\left ( \begin{array}{c}
x_1^2\\ x_2^2\\ x_3^2 \end{array}\right )-
\left ( \begin{array}{c}
1\\ 1\\ -1 \end{array}\right )\Big ] \\
&=& \frac 12 \left ( \begin{array}{c}
x_1\\ x_2\\ x_3 \end{array}\right )+\frac 12 \left ( \begin{array}{c}
x_1^{-1}\\ 0\\ 0\end{array}\right ).
\end{eqnarray*}
It implies
\[
x^+-x^*= \frac 12 \left ( \begin{array}{c}
(x_1-x_1^*) + (x_1^{-1}-(x_1^*)^{-1}) \\ x_2-x_2^*\\ x_3-x_3^* \end{array}\right ),
\]
which shows the linear convergence of $\{x_k\}$.

On the other hand the tensor equation (\ref{ex-1}) can be written as
\[
A \left ( \begin{array}{c}
x_1^2\\x_2^2\\ x_3^2\end{array}\right )
=b,\quad \mbox{with }\quad A=
\left ( \begin{array}{ccc}
1 & 0 & 0 \vspace{1mm} \\
1 & 1 & 0 \vspace{1mm} \\
-1& 1 & 1
\end{array}\right ), \; b=\left ( \begin{array}{c}
1\\1\\ -1\end{array}\right ).
\]
It can be solved easily by solving the system of linear equation
$Ay=b$ first and then get the solution of the tensor equation by
letting $x=\pm y^{1/2}$.

\noindent
{\bf Example 1.2}
Let consider a more general equation:
\[
F(x)=\frac {1}{m-1} M x^{[m-1]}- b=0,
\]
where  $M\in R^{n\times n}$ is nonsingular, $b\in R^n$ and $x^{[m-1]}=(x_1^{m-1}, \ldots, x_n^{m-1})^T$.
It is a tensor equation in the form (\ref{multi-linear equation}) where the elements of $\cal A$ are
\[
a_{ij\ldots j}=m_{ij},\quad i,j=1,2,\ldots,n
\]
and all other elements of $\cal A$ are zero. It is easy to see that the Jacobian of $F$
at $x$ is
\[
F'(x)= M \; \mbox{diag}(x_1^{m-2}, x_2^{m-2},\ldots x_n^{m-2}).
\]
It is singular  even if $M$ is nonsingular, if there exists an $x_i=0$.
As a result, Newton's method may fail to work if some iterate $x_k$ has zero elements.

On the other hand, if $M$ is nonsingular, the equation can be solved easily by solving the system
$\frac{1}{m-1}My=b$ first and then letting $x=y^{[1/(m-1)]}$.

The last two examples indicate that it is important to develop
special iterative methods for solving the tensor equation  (\ref{multi-linear equation})
using the special structure of the tensor $\cal A$. It is the major purpose of the paper.

In this paper, we will pay particular attention to tensor equations (\ref{multi-linear equation}) where
the coefficient tensor ${\cal A}$ is an M-tensor.
Tensor ${\cal A}=(a_{i_1i_2\ldots i_m})\in {\cal T}(m,n)$ is called a Z-tensor if it can be written as
\begin{equation}\label{M-tensor}
{\cal M}=s{\cal I}-{\cal B},
\end{equation}
where $s>0$, $\cal I$ is the identity tensor whose diagonals are all ones and all off diagonal elements
are zeros, and ${\cal B}\ge 0$ is a nonnegative tensor in the sense that all its elements are nonnegative.
Tensor ${\cal A}$ is called M-tensor
if it can be written as (\ref{M-tensor}) and satisfies $s\ge \rho ({\cal B})$,
where $\rho ({\cal B})$ is the spectral radius of tensor ${\cal B}$, that is
\[
\rho(\cal B)= \max \left\{\left| \lambda \right|: \lambda \text{ is an eigenvalue of } \cal{B}\right \}.
\]
If $s> \rho ({\cal B})$, then ${\cal A}$ is called a strong or nonsingular M-tensor \cite{Ding-Qi-Wei-13}.
In the case  ${\cal A}$ is an M-tensor, we call the equation (\ref{multi-linear equation}) M-tensor
equation and abbreviate it as M-TEQ.
An interesting property of the M-TEQ is that if ${\cal A}$ is a strong M-tensor, then for every positive vector $b$
the tensor equation (\ref{multi-linear equation}) has a unique positive solution, namely, all entries of the solution are positive \cite{Ding-Wei-16}.

The study in the numerical methods for solving tensor equation has begun only a few years ago. Most of
them focus on solving the M-TEQ.
 Ding and Wei \cite{Ding-Wei-16} extended the classical iterative methods for solving system of linear equations
 including the Jacobi method,
 the Gauss-Seidel method, and successive overrelaxation method to solve the M-TEQ. The
 methods are linearly convergent to a nonnegative solution of the equation.
 Han \cite{Han-17} proposed a homotopy method for finding the unique positive solution of the M-TEQ with strong M-tensor
 and a positive right side vector $b$.
 Liu, Li and Vong \cite{Liu-Li-Vong-18} proposed a tensor splitting method for solving the M-TEQ.
 An advantage of the method is that at each iteration, only a system of linear equation needs to be solved.
 The coefficient matrix of the linear system does not depend on the iteration.
 The sequence of iterates $\{x_k\}$  generated by the method converges to a nonnegative solution
 of the equation monotonically if the initial point $x_0$ is restricted to the set
 $\{x\in R^n\;|\; 0<{\cal A}x^{m-1}\le b\}$.
 Quite recently,
He, Ling, Qi and Zhou \cite{He-Ling-Qi-Zhou-18} proposed a Newton-type method to solve the M-TEQ and
established its quadratic convergence.
Li, Xie and Xu \cite{Li-Xie-Xu-17} also extended the classic splitting methods for solving system of linear equations
to solving symmetric tensor equations. The methods in \cite{Li-Xie-Xu-17} can
solve those equations (\ref{multi-linear equation}) where ${\cal A}$ are not necessary M-tensors.
Li, Dai and Gao \cite{Dai-2019} developed an alternating projection method for a class of tensor equations
and established its global convergence and linear convergence.
Related work can also be found in \cite{4, Li-Ng-15, Lv-Ma-18,Xie-Jin-Wei-2017a, Xie-Jin-Wei-2017,Yan-Ling-18}.
Just when this article is about to be completed, we knew the new report by Bai, He, Ling and Zhou \cite{Bai-He-Ling-Zhou-18}
where the authors proposed an iterative method for finding a nonnegative solution to M-tensor equations.

The existing methods for solving M-TEQ focus on finding a positive solution of the equation under
the restriction $b>0$.
 In this paper, we further study numerical methods for solving M-TEQ (\ref{multi-linear equation}).
Our purpose is to find the largest nonnegative solution of the equation without the requirement of $b>0$.
We first give a necessary and sufficient condition for the existence of the nonnegative solution of the
M-TEQ. The result reals an interesting property that
if an M-TEQ has a nonnegative solution, then it has a largest nonnegative solution.

We then propose an iterative method to solve the M-TEQ with strong M-tensor.
We first split the coefficient M-tensor into two parts. Based on the splitting form, we develop an
 approximation Newton method. At each iteration, we solve a system of linear equations.
 The coefficient matrix of linear system is independent of the iteration.
  We show that if the initial point is appropriately chosen, then the generated sequence of iterates
  converges to a nonnegative solution of the equation without the restriction that $b$ is positive.
   We also do
 numerical experiments to test the proposed method. The results show that the methods are very efficient.

The rest of this paper is organized as follows. In Section 2, we derive
a necessary and sufficient condition for an M-tensor equation to have a nonnegative
solution. It particularly shows an interesting result that if an M-tensor equation has more than one nonnegative solutions,
then it has a largest nonnegative solution.
In Section 3, we  propose a monotone iterative method
for solving the M-TEQ with strong M-tensor and establish its monotone
 convergence. In Section 4,
we make an improvement to the method proposed  in Section 3 and establish its monotone convergence.
At last, we do some numerical experiments to test the proposed method in Section 5.
We conclude the paper by giving some final remark.

\section{The Existence of Nonnegative Solutions to the M-Tensor Equation}
\setcounter{equation}{0}

Consider the tensor equation
\begin{equation}\label{eqn:tensor}
F(x)={\cal M}x^{m-1}-b=0,
\end{equation}
where ${\cal M}=(m_{i_1i_2\ldots i_m})\in {\cal T}(m,n)$ and $b\in R^n$.
We are interested in the nonnegative solution of an M-tensor equation.
It was proved by Ding and Wei \cite{Ding-Wei-16} that if $\cal M$ is a strong M-tensor and
$b$ is positive,
then the equation has unique positive solution.
Liu, Li and Vong \cite{Liu-Li-Vong-18} also obtained some existence and
uniqueness of the positive/nonnegative solutions for the M-TEQ via tensor splitting technique.

Similar to Theorem 5.4 in \cite{Gowda-Luo-Qi-Xiu-15}, we now derive a sufficient condition
for the existence and uniqueness of the nonnegative solutions to an M-TEQ.
We first introduce the concept of the so called majorization matrix of a tensor
${\cal M}\in {\cal T}(m,n)$ \cite {Li-Ng-14,Liu-Li-Vong-18,Pearson-10}.
It is a matrix $M=(m_{ij})\in R^{n\times n}$ whose elements are
\begin{equation}\label{def:M}
m_{ij}=m_{ij\ldots j}, \quad i,j=1,2,\ldots,n.
\end{equation}

The theorem about the existence and uniqueness of the nonnegative solution of an M-TEQ is stated and proved below.
\begin{theorem}\label{th:unique}
Suppose ${\cal M} = (m_{i_1\ldots i_m})\in {\cal T}(m,n)$ is a strong M-tensor
such that for each index $i$,
\[
m_{i i_2\ldots i_m} = 0,\quad \mbox{whenever }\; i_j\neq i_k \; \;\mbox {for some } j\neq k.
\]
Then, the M-TEQ has a unique nonnegative solution if and only if the unique solution
of the M-LEQ
\[
My=b
\]
is nonnegative, where matrix $M=(m_{ij})\in R^{n\times n}$ is the majorization matrix of $\cal M$.
\end{theorem}
\begin{proof}
It is not difficult to see that the M-TEQ can be written as the following equation
\[
M x^{[m-1]}-b=0.
\]
Since $\cal M$ is a strong M-tensor, its majorization matrix $M$ is a strong M-matrix \cite {Liu-Li-Vong-18}.
It is nonsingular and satisfies $M^{-1}\ge 0$.
The results of the theorem becomes obvious.
\end{proof}

In many cases where $b$ is not positive,
the M-tensor equation may have no nonnegative solutions or multiple nonnegative solutions.
For  example, if $m=2$ and $b<0$, then for any
strong M-matrix $M$, the unique solution of the equation (\ref{eqn:tensor}) is negative.
Consequently, it has no nonnegative solution.
The following example suggested by Prof. C. Ling via private communication shows that
an M-TEQ with a strong M-tensor $\cal M$ may have more than one nonnegative solutions.

\noindent
{\bf Example 2.1} Given tensor ${\cal M}\in {\cal T}(4,2)$ in the form ${\cal M}=3 {\cal I}-{\cal B}$,
where $\cal I$ is the identity tensor and $\cal B$ is a nonnegative tensor  with elements
\[
b_{1111}=b_{2222}=0,\quad b_{1122}=\frac 32,\quad b_{1222}=\frac 12
\]
and all other elements of ${\cal B}$ are zeros. It is easy to see that the unique eigenvalue of $\cal B$ is zero.
Consequently,  $\cal M$ is an M-tensor. Let $b=(-7, 24)^T$. Then the tensor equation
${\cal M}x^{m-1}-b=0$ is written as
\[\left \{\begin{array}{lll}
3x_1^3-\frac 32 x_1x_2^2-\frac 12 x_2^3 =-7,\\
3x_2^3=24.
\end{array}\right.
\]
Clearly, the solutions of the following system are all the solutions of the M-TEQ:
\[\left \{\begin{array}{lll}
x_1^3- 2 x_1+ 1 =0,\\
x_2=2.
\end{array}\right.
\]
The last system has two nonnegative solutions
\[
x^{(1)}=(1,2)^T\quad \mbox{and}\quad x^{(2)}=(\frac {\sqrt 5-1}{2},2)^T.
\]

The following example shows that an M-tensor equation with an odd order strong M-tensor may also have
multiple nonnegative solutions.

\noindent
{\bf Example 2.2} Given tensor ${\cal M}\in {\cal T}(3,2)$ in the form ${\cal M}= {\cal I}-{\cal B}$,
where $\cal I$ is the identity tensor and $\cal B$ is a nonnegative tensor  with elements
\[
b_{111}=b_{222}=0,\quad b_{112}=\frac 32 ,\quad b_{122}=1
\]
and all other elements of ${\cal B}$ are zeros. It is easy to see that the unique eigenvalue of $\cal B$ is zero.
Consequently,  $\cal M$ is a strong M-tensor. Let $b=(-6, 4)^T$. Then the tensor equation
${\cal M}x^{m-1}-b=0$ is written as
\[\left \{\begin{array}{lll}
x_1^2-  \frac 32 x_1x_2- x_2^2 =-6,\\
x_2^2=4.
\end{array}\right.
\]
The last system has two nonnegative solutions
\[
x^{(1)}=(1,2)^T\quad \mbox{and}\quad x^{(2)}=(2,2)^T.
\]

Recently, Xu, Gu and Huang \cite{Huang-18} investigated some interesting properties of
the tensor equations. In this section, we investigate the existence of the nonnegative solutions
and the largest nonnegative solution of the Z/M-tensor equation.

To this end, we first introduce the concept of the so-called joint semi-sublattice.
A set ${\cal S}\subseteq R^n$ is called a join semi-sublattice if $x,y\in {\cal S}$, then there join, i.e.,
$\max \{x,y\}$, is also in ${\cal S}$.

The following theorem can be regarded as a dual theorem of Theorem 3.11.5 in \cite{Cottle-Pang-Stone-09}.

\begin{theorem}\label{th:lattice}
If ${\cal S}$ is a nonempty and bounded closed join semi-sublattice, then it has a largest element.
\end{theorem}
\begin{proof}
Given an arbitrary positive vector $c$ and any $\bar x\in {\cal S}$.
Consider the following program problem:
\[\left \{ \begin{array}{cll}
\max & \phi (x)=c^Tx \\
\mbox{s.t.} & x\in {\cal S},\; \; c^Tx \ge c^T \bar x.
\end{array}\right.
\]
Clearly, the feasible set of the problem is closed and bounded from above. Hence, the problem has a solution $\tilde x$.
For any $x\in {\cal S}$, we have $z=\max\{x,\tilde x\}\in {\cal S}$. So, we have
\[
c^T (\tilde x- z)\ge 0.
\]
Since $\tilde x\le z$ and $c$ is positive, the last inequality implies
$\tilde x=z$.
By the arbitrariness of $x$, we claim that $\tilde x$ is the largest element of ${\cal S}$.
\end{proof}

Define the set
\begin{equation}\label{feasible}
{\cal S}=\{x\in R_+^n\;|\; {\cal M}x^{m-1}\le b\}=\{x\in R_+^n\;|\; F(x)\le 0\}.
\end{equation}
It looks very similar to the feasible set
\[
\tilde {\cal S}=\{x\in R_+^n\;|\; {\cal M}x^{m-1}\ge b\}=\{x\in R_+^n\;|\; F(x)\ge 0\}
\]
of the tensor complementarity problem
\[
x\ge 0\quad F(x)\ge 0,\quad x^TF(x)=0.
\]
It was proved by Luo, Qi and Xiu \cite{Luo-Qi-Xiu-17} that if $\cal M$ is a Z-tensor, then the feasible
set $\tilde {\cal S}$  has a least element. Moreover, the least element of $\tilde {\cal S}$
is a solution of the tensor complementarity problem.
In what follows, we will establish a similar result to the M-TEQ.
Specifically, we will show that if $\cal M$ is a strong  M-tensor,
then $\cal S$ has a largest element as long as it is not empty. Moreover, the largest element of the set is the
largest nonnegative solution to the M-TEQ.

Denote for $x\in R^n$ and real number $k$
\[
x^{[k]}=(x_1^k, x_2^k,\ldots, x_n^k)^T.
\]

The lemma bellow show that $\cal S$ has a largest element.

\begin{lemma}\label{lm:M-tensor}
If $\cal M$ is a strongly  M-tensor, and the set $\cal S$ defined by (\ref{feasible})
is not empty, then $\cal S$ is a bounded joint semi-sublattice. As a result, $\cal S$ has a largest element.
\end{lemma}
\begin{proof}
We first show that ${\cal S}$ is a joint semi-sublattice. Indeed, we have for any  $x,y\in {\cal S}$, let $z=\max \{x,y\}$.
We have for each $i=1,2,\ldots,n$, if $z_i=x_i$,
\begin{eqnarray*}
({\cal M}z^{m-1})_i &=& \sum_{i_2, \ldots, i_m}m_{ii_2\ldots i_m}z_{i_2}\cdots z_{i_m}\\
    &=&m_{ii\ldots i}x_i^{m-1} + \sum_{(i_2, \ldots, i_m)\neq (i, \ldots, i)}m_{ii_2\ldots i_m}z_{i_2}\cdots z_{i_m}\\
    &\le & m_{ii\ldots i}x_i^{m-1} + \sum_{(i_2, \ldots, i_m)\neq (i, \ldots, i)}m_{ii_2\ldots i_m}x_{i_2}\cdots x_{i_m} \\
    &=& ({\cal M}x^{m-1})_i\le b_i,
\end{eqnarray*}
where the first inequality follows from the fact
$m_{i i_2\ldots i_m}\le 0$, $\forall (i, i_2,\ldots, i_m)\neq (i,i,\ldots, i)$.
Similarly, we can show that $({\cal M}z^{m-1})_i\le b_i$ for $z_i=y_i$.
Consequently, $\cal S$ is a joint semi-sublattice.

Next, we show that $\cal S$ is bounded. Suppose on the contrary that there is an unbounded sequence $\{x_k\}\subset {\cal S}$.
That is, $\|x_k\|_{\infty} \to\infty $ as $k\to\infty$ and
\[
x_k\ge 0,\quad {\cal M}x_k^{m-1} -b\le 0.
\]
Denote $u_k=x_k/\|x_k\|_{\infty}$. Then $\{u_k\}$ is bounded and hence has a limit point.
Without loss of generality, we let $\{u_k\}\to u$. It is easy to get
\[
u\ge 0,\qquad {\cal M}u^{m-1}\le 0.
\]
We write ${\cal M}=s{\cal I}-{\cal B}$ with $s>\rho ({\cal B})$.
It follows from the last inequality that
\[
s u_i^{m-1} \le \Big ({\cal B}u^{m-1}\Big )_i,\quad \forall i=1,2,\ldots,n,
\]
which implies
\[
s\le \min _{u_i\neq 0}\frac {\Big ({\cal B}u^{m-1}\Big )_i}{u_i^{m-1}}\le \rho ({\cal B}),
\]
where the last inequality follows from \cite [Theorem 3.25]{Qi-Luo-17}.
The last equality yields a contradiction. Consequently, $\cal S$ is bounded.

At last, by Theorem \ref{th:lattice}, $\cal S$ has a largest element.
The proof is complete.
\end{proof}

Based on the last lemma and Theorem \ref{th:lattice},   we can derive a condition for the existence of nonnegative
solution of the M-TEQ.

\begin{theorem}\label{th:positive}
Let $\cal M$ be a strongly M-tensor. Then the following statements are true.
\begin{itemize}
\item [(i)] The M-tensor equation (\ref{eqn:tensor}) has a nonnegative
solution if and only if the set ${\cal S}$  defined by
(\ref{feasible}) is not empty. In the case ${\cal S}\neq\emptyset $, its largest element
is the largest nonnegative solution of the equation. In particular, if $M^{-1}b\ge 0$,
then (\ref{eqn:tensor}) has a nonnegative solution.
\item [(ii)] The equation (\ref{eqn:tensor}) has a positive
solution if and only if the set ${\cal S}$ contains a positive element.
In particular, if $M^{-1}b>0$, then (\ref{eqn:tensor}) has a positive
solution.
\end{itemize}
\end{theorem}
\begin{proof}
Conclusion (ii) follows from conclusion (i) directly. We only need to verify (i).

Since the solution set of (\ref{eqn:tensor}) is a subset  of $\cal S$ defined by (\ref{feasible}),
`only if' part is obvious. We turn to the `if' part.

Suppose  ${\cal S}\neq\emptyset $. Theorem \ref{th:lattice} and Lemma \ref{lm:M-tensor}
ensure that $\cal S$ has a largest element $\bar x$, i.e.,
\[
x\le \bar x\quad\forall x\in {\cal S}.
\]
It suffices to show that $\bar x$ is a solution of (\ref{eqn:tensor}).

Suppose on the contrary that $\bar x$ is not a solution of (\ref{eqn:tensor}).
There must be an index $i$ such that
\[
F_i(\bar x)=({\cal M}\bar x^{m-1}-b)_i<0.
\]
Let $\cal D$ be the diagonal part subtensor of ${\cal M}$ and
${\cal B}={\cal M}-{\cal D}$. By the definition of M-tensor, we have ${\cal B}\le 0$.
Let $D$ be a diagonal matrix whose diagonals are diagonals of $\cal M$ or $\cal D$, i.e.,
\[
D=\mbox{diag}(m_{1\ldots 1},\ldots, m_{n\ldots n}).
\]
Define $\tilde x$ with elements
\[
\tilde x_j= \Big (\bar x_j^{m-1}-\frac 12 D^{-1} F_j(\bar x) \Big )^{1/(m-1)},\quad j=1,2,\ldots,n.
\]
It implies
\[
D\tilde x ^{[m-1]}= D\bar x^{[m-1]}-\frac 12 F(\bar{x}).
\]
It is easy to see that $0\le \bar x\le \tilde x$ and
\begin{eqnarray*}
F(\tilde x) &=& {\cal D}\tilde x^{m-1}+{\cal B} \tilde x^{m-1} -b \\
        &\le & {\cal D}\tilde x^{m-1}+{\cal B} \bar x^{m-1} -b \\
        &=& D\tilde x^{[m-1]} + {\cal B} \bar x^{m-1} -b \\
        &=& D \bar x^{[m-1]} -\frac 12 F(\bar x)  + {\cal B} \bar x^{m-1} -b\\
        &=& \frac 12 F(\bar x) \le 0.
\end{eqnarray*}
Consequently, $\tilde x\in {\cal S}$. However, by the definition of $\tilde x$, we obviously have
$\tilde x\ge \bar x$. Moreover, there are at least one indices $i$ such that
$\tilde x_i>\bar x$, which contracts the fact that $\bar x$ is the largest element of $\cal S$.
The contradiction shows that $\bar x$ must be a solution of (\ref{eqn:tensor}).

Since the solution set of (\ref{eqn:tensor}) is a subset of $\cal S$, the largest element of
$\cal S$ is the largest solution of (\ref{eqn:tensor}). The proof is complete.
\end{proof}

\section{A Monotone Iterative Method for Finding a Nonnegative Solution to the M-Tensor Equation}
\setcounter{equation}{0}

In this section, we develop an iterative method for solving the M-tensor
equation. Our purpose is to find a nonnegative solution of the equation.
So we assume throughout this section that tensor $\cal M$ in (\ref{eqn:tensor}) is a strong M-tensor and
that the set ${\cal S}$ defined by (\ref{feasible}) is not empty.

The method can be regarded as an approximation to Newton's method.
Notice that for each $i=1,2,\ldots, n$, function $({\cal M}x^{m-1})_i$ is a
homogeneous polynomial. It can be written as ${\cal M}_ix^{m-1}$ with symmetric tensor
${\cal M}_i\in {\cal ST}(m-1,n)$.

Let us consider the Newton iteration:
\[
(m-1){\cal M}x_k^{m-2}(x_{k+1}-x_k)+ F(x_k)=0,\quad k=0,1,2,\ldots.
\]
An attractive property of Newton's method is its quadratic convergence if
the Jacobian at the solution $x^*$, $F'(x^*)=(m-1){\cal M}(x^*)^{m-2}$ is nonsingular.
However,  the method may be failure if ${\cal M}x_k^{m-2}$ is singular.

We are going to develop an approximation Newton method. To this end, we
split the tensor ${\cal M}$ as
\[
{\cal M}=\widetilde{{\cal M}}+\overline {\cal M},
\]
where $\widetilde{{\cal M}}$ is the sub-tensor of ${\cal M}$  whose possibly nonzero elements are
$m_{ij\ldots j}$, $i,j=1,2,\ldots, n$ while other elements are zeros,
$\overline {\cal M}={\cal M}-\widetilde{{\cal M}}$.
The sub-tensor $\widetilde {\cal M}$ contains all diagonal elements of ${\cal M}$.
So, $\overline {\cal M}$ is a non-positive tensor, i.e., $\overline {\cal M}\le 0$.

Let $M=(m_{ij})$ be majorization matrix defined by (\ref{def:M}).
It is an (a strong) M-matrix if ${\cal M}$ is an (a strong) M-tensor \cite {Liu-Li-Vong-18}.

Using the above splitting form to ${\cal M}$, we write the Newton iteration as
\begin{eqnarray*}
0 &=& {\cal M}x_k^{m-2}x_{k+1}-{\cal M}x_k^{m-1}+ \frac {1}{m-1} F(x_k)\\
    &=& M x_{k+1} ^{[m-1]} - Mx_k^{[m-1]} + \frac {1}{m-1} F(x_k)
        + [{\cal M}x_k^{m-2}x_{k+1}-Mx_{k+1}^{[m-1]}
    - \overline {\cal M} x_k^{m-1}]\\
    &=& M x_{k+1} ^{[m-1]} - Mx_k^{[m-1]} + \frac {1}{m-1} F(x_k)
       + \widetilde {\cal M}(x_k^{m-2}-x_{k+1}^{m-2}) x_{k+1}
    + \overline {\cal M} x_k^{m-2}( x_{k+1}-x_k)\\
    & = &  M x_{k+1} ^{[m-1]} - Mx_k^{[m-1]} + \frac {1}{m-1} F(x_k) + O(\|x_{k+1}-x_k\|),\quad k=0,1,\ldots.
\end{eqnarray*}

By neglecting the term $O(\|x_{k+1}-x_k\|)$, we  get an approximate Newton iterative scheme
\[
x_{k+1}^{[m-1]}=x_k^{[m-1]}- \frac {1}{m-1}M^{-1}F(x_k),\quad k=0,1,\ldots .
\]
More generally, we give the following iterative scheme:
\[
x_{k+1}^{[m-1]}=x_k^{[m-1]}- \alpha_k M^{-1}F(x_k),\quad k=0,1,\ldots,
\]
or
\begin{equation}\label{iter:teq-1}
x_{k+1}^{[m-1]}=x_k^{[m-1]}+\alpha _k d_k,\quad Md_k+ F(x_k)=0,
\end{equation}
where $\alpha_k \in (0,1]$. We call the above iterative scheme
sequential M-matrix equation method (S-MEQM) because
the subproblem (\ref{iter:teq-1}) is a system of linear equation in $d_k$ with an M-matrix
as the coefficient matrix. Since $M$ is independent of
the iteration $k$, we can solve the system by LU decomposition at the beginning of
the method. So, at each iteration, the computation cost for
solving the system of linear equations (\ref{iter:teq-1}) is $O(n^2)$
except for the initial step.

Recently Liu, Li and Vong \cite{Liu-Li-Vong-18}
proposed a so-called tensor splitting (TS) method for solving the M-tensor equation.
The TS method corresponds to $\alpha_k\equiv 1$. It is monotonically and linearly convergent if the initial point
$x_0$ is chosen in the set
 \begin{equation}\label{feasible-1}
{\cal S}_1=\{x\in R_+^n\;|\; 0<{\cal M}x^{m-1}\le b\}.
\end{equation}
Specifically, they proved the following theorem  \cite{Liu-Li-Vong-18}.
\begin{theorem}
Let ${\cal M}$ be a strong M-tensor and $b$ be positive.
Then the sequence $\{x_k\}$ generated by the TS method with $x_0\in {\cal S}_1$
converges to the unique positive solution of (\ref{eqn:tensor}) monotonically in the sense
\[
x_{k+1}\ge x_k\ge 0.
\]
Moreover, the convergence rate of $\{x_k\}$ is linear.
\end{theorem}

The condition $x_0\in {\cal S}_1$ in the above convergence theorem is restrictive.
It makes the method suitable for those
problems with $b>0$ only.
The theorem below shows that the  S-MEQM
will retain monotone convergence if the initial point is in ${\cal S}$.

\begin{theorem}\label{th:conv-1}
Starting from any $x_0\in {\cal S}$, the sequence $\{x_k\}$ generated by the S-MEQM (\ref{iter:teq-1})
is contained in ${\cal S}$ and satisfies
\begin{equation}\label{monotone}
 x_{k+1}\ge x_k\ge 0,\quad k=0,1,2,\ldots .
\end{equation}
Moreover, $\{x_k\}$ converges to a nonnegative solution of the equation (\ref{eqn:tensor}).
\end{theorem}
\begin{proof}
We first prove $\{x_k\}\subset {\cal S}$ and (\ref{monotone})  by induction.

For $k=0$, we have $F(x_0)\le 0$ and
\[
Mx_1^{[m-1]}=Mx_0^{[m-1]} - \alpha_0 F(x_0)\ge Mx_0^{[m-1]}.
\]
Since $M$ is  an M-matrix, the last inequality yields  (\ref{monotone}) with $k=0$.
Moreover, we have
\begin{eqnarray*}
F(x_1) &=& {\cal M}x_1^{m-1}-b = Mx_1^{[m-1]} + \overline {\cal M} x_1^{m-1}-b\\
    &=& Mx_0^{[m-1]}- \alpha_0 F(x_0)+ \overline  {\cal M} x_1^{m-1}-b\\
    &=&  ( 1-\alpha_0 ) F(x_0) +\overline  {\cal M} (x_1^{m-1}-x_0^{m-1})\\
    &\le & ( 1-\alpha_0 ) F(x_0)\le 0,
\end{eqnarray*}
where the inequality holds because $x_1\ge x_0$ and $\overline {\cal M}\le 0$.

Suppose the inequalities $x_k\ge x_{k-1}\ge 0$ and $F(x_k)\le 0$
hold for some $k\ge 1$.
It follows from (\ref{iter:teq-1}) that
\[
x_{k+1}^{[m-1]}=x_k^{[m-1]} -\alpha_k M^{-1}F(x_k) \ge x_k^{[m-1]}.
\]
Moreover, we can get
\[
F(x_{k+1})= Mx_{k+1}^{[m-1]}+\overline {\cal M}x_{k+1}^{m-1}-b\le
     ( 1-\alpha_k ) F(x_k) \le 0.
\]
By the principle of induction, we claim that the inequalities
in (\ref{monotone}) hold and $F(x_k)\le 0$ for all $k\ge 0$.

Since $\{x_k\}$ is bounded from above, then $\{x_k\}$ converges. Taking limits
in both sizes of (\ref{iter:teq-1}), it is easy to see that the limit of
$\{x_k\}$ is a nonnegative solution of the tensor equation (\ref{eqn:tensor}).
The proof is complete.
\end{proof}

We turn to the convergence rate of the method. Suppose $\{x_k\}\to \bar x$.
Clearly, we have $x_k\le x_{k+1}\le \bar x$. If for some $i$, $\bar x_i=0$, then
$(x_k)_i=0$ for all $k\ge 0$. So, without loss of generality, we assume $\bar x>0$.

Define
\begin{equation}\label{def:phi}
\phi (x)= \Big ( x ^{[m-1]}- \alpha M^{-1}F(x)\Big )^{1/(m-1)}.
\end{equation}
The iterative scheme can be written as
\[
x_{k+1}=\phi (x_k)
\]
and the limit point $\bar x$ satisfies
\[
\bar x=\phi (\bar x).
\]

It follows from (\ref{def:phi}) that
\[
M\phi (x)^{[m-1]}= M x ^{[m-1]}- \alpha F(x).
\]
Without loss of generality, we suppose that ${\cal M}$ is symmetric.
Taking derivative in both sizes of the equality, we obtain
\begin{eqnarray*}
M \mbox{ diag} (\phi (x) ^{[m-2]}) \phi '(x) &=& M \mbox{ diag}(x^{[m-2]}) - \frac {1}{m-1} \alpha F'(x) \\
    &=& M \mbox{ diag}(x^{[m-2]}) -  \alpha {\cal M}x^{m-2}.
\end{eqnarray*}
At the limit $\bar x$,  the last equality yields
\[
M \mbox{ diag} ( \bar x ^{[m-2]}) \phi '(\bar x) = M \mbox{ diag}( \bar x^ {[m-2]}) -  \alpha {\cal M}\bar x^{m-2}.
\]
It implies
\begin{eqnarray*}
\phi' (\bar x) &=& I- \alpha \mbox{ diag} (\bar x^{[-(m-2)]}) M^{-1} {\cal M}\bar x^{m-2}\\
    &=& (1-\alpha )I - \alpha \mbox{ diag} (\bar x^{[-(m-2)]}) M^{-1} \overline {\cal M}\bar x^{m-2}\ge 0.
\end{eqnarray*}
If $M^{-1}b>0$, then it follows from the last equality that
\[
\phi' (\bar x) \bar x = \bar x - \alpha \mbox{ diag} (\bar x^{[-(m-2)]}) M^{-1} {\cal M}\bar x^{m-1}
=\bar x - \alpha \mbox{ diag} (\bar x^{[-(m-2)]})M^{-1}b<\bar x.
\]
Since $\phi' (\bar x)$ is nonnegative, we claim that the spectral radius of $\nabla \phi'(\bar x)$ satisfies
$\rho (\nabla \phi'(\bar x))<1$. As a result, the convergence of the sequence $\{x_k\}$ is linear.
Theorem 4.5 in \cite{Liu-Li-Vong-18} shows that the condition $M^{-1}b>0$ implies that the
the M-TEQ has a unique positive solution.

The above arguments have shown the  following theorem.

\begin{theorem}\label{th:rate}
Suppose that $\cal M$ is a strong M-tensor and $M$ is its majorization matrix.
If $M^{-1}b>0$, then the convergence rate of $\{x_k\}$ generated by  the S-MEQM (\ref{iter:teq-1})
converges to the unique positive solution of the M-tensor equation (\ref{eqn:tensor}) linearly.
\end{theorem}

Let $D$, $-L$ and $-U$ be the diagonal part, strict lower triangular part and strict upper triangle part of
the matrix $M$, i.e.,
\[
M=D-L-U.
\]
Similar to the splitting methods for solving system of linear equations,
on the basis of (\ref{iter:teq-1}), we can develop splitting type methods.

\begin{itemize}
\item The Jacobi iteration:
\begin{equation}\label{iter:Jacobi}
x^{[m-1]}_{k+1} = x_k^{[m-1]}-\alpha_k D^{-1}F (x_k),\;k=0,1,2,\ldots.
\end{equation}
This iterative scheme is the same as the Jacobi method by Ding and Wei \cite{Ding-Wei-16}.
\item The Gauss-Seidal iteration:
\begin{equation}\label{iter:GS}
x^{[m-1]}_{k+1} = x_k^{[m-1]}- \alpha_k (D-L)^{-1} F(x_k) ,\;k=0,1,2,\ldots.
\end{equation}
\item Successive over-relaxation (SOR) iteration:
\begin{equation}\label{iter:SOR}
x^{[m-1]}_{k+1} = x_k^{[m-1]}- \alpha_k \omega (D-\omega L)^{-1} F(x_k) ,\;k=0,1,2,\ldots.
\end{equation}
\end{itemize}

Similar to  Theorem \ref{th:conv-1}, it is not difficult
to establish the monotone convergence of the above iterative methods.
As an example, we derive the convergence of the Gauss-Seidal method below.

\begin{theorem}\label{th:conv-3}
Let the sequence $\{x_k\}$ be generated by the Gauss-Seidal method (\ref{iter:GS}) with initial point
$x_0\in {\cal S}$. Then $\{x_k\}\subset {\cal S}$ and converges to a nonnegative solution of
(\ref{eqn:tensor}) monotonically.
\end{theorem}
\begin{proof}
For $k=0$, we have $F(x_0)\le 0$ and
\[
x_1^{[m-1]}=x_0^{[m-1]} - \alpha_0 (D-L)^{-1}F(x_0)\ge x_0^{[m-1]}.
\]
It implies
\begin{eqnarray*}
F(x_1) &=& {\cal M}x_1^{m-1}-b = Mx_1^{[m-1]} + \overline {\cal M} x_1^{m-1}-b\\
    &=& (D-L-U)x_1^{[m-1]}+\overline {\cal M} x_1^{m-1}-b\\
    &=& (D-L)x_1^{[m-1]}-Ux_1^{m-1}+\overline {\cal M} x_1^{m-1}-b\\
    &=& (D-L)x_0^{[m-1]}-\alpha_0 F(x_0)-Ux_1^{m-1}+\overline {\cal M} x_1^{m-1}-b\\
    &=& (1-\alpha_0)F(x_0)-U(x_1^{[m-1]}-x_0^{[m-1]})+\overline {\cal M} (x_1^{m-1}-x_0^{m-1})\\
    &\le & 0,
\end{eqnarray*}

Suppose the inequalities $x_k\ge x_{k-1}\ge 0$ and $F(x_k)\le 0$
hold for some $k\ge 1$. We can get for all $\alpha_k \in (0,1]$,

\[
x_{k+1}^{[m-1]}=x_k^{[m-1]} - \alpha_k (D-L)^{-1}F(x_k)\ge x_k^{[m-1]}
\]
and
\begin{eqnarray*}
F(x_{k+1}) &=& Mx_{k+1}^{[m-1]} + \overline {\cal M}x_{k+1}^{m-1}-b \\
    &=& (D-L) x_{k+1}^{[m-1]} -U x_{k+1}^{[m-1]}+ \overline {\cal M}x_{k+1}^{m-1}-b\\
     &=& (D-L)x_k^{[m-1]}-\alpha_k F(x_k) -U x_{k+1}^{[m-1]}+ \overline {\cal M}x_{k+1}^{m-1}-b\\
     &=&  Mx_k^{[m-1]} - \alpha_k F(x_k)-U(x_{k+1}^{[m-1]}-x_k^{[m-1]})+ \overline {\cal M}x_{k+1}^{m-1}-b \\
     &=& (1-\alpha_k )F(x_k)-U(x_{k+1}^{[m-1]}-x_k^{[m-1]}) + \overline {\cal M}(x_{k+1}^{m-1}-x_k^{m-1})\le 0.
\end{eqnarray*}

By induction, it is not difficult to show that if $x_0\ge 0$ and $F(x_0)\le 0$,
then the following inequalities hold for all $k\ge 0$:
\[
 x_{k+1}^{[m-1]}\ge x_k^{[m-1]}\ge 0,\quad
F(x_k)\le 0.
\]
Since $\{x_k\}$ is bounded from above, we claim that $\{x_k\}$ converges and hence $\{F(x_k)\}$ converges to 0.
The proof is complete.
\end{proof}

\section{Improvement}
\setcounter{equation}{0}

In this section, we make some improvement to the monotone method proposed in the last section.
We rewrite the Newton method as
\[
0={\cal M}x_k^{m-2}(x_{k+1}-x_k)+\frac 1{m-1}F(x_k)
    =Mx_{k+1}^{[m-1]}-Mx_k^{[m-1]} + \frac 1{m-1}F(x_k) +r_k,
\]
where
\[
r_k= {\cal M}x_k^{m-2}(x_{k+1}-x_k)- (Mx_{k+1}^{[m-1]}-Mx_k^{[m-1]} ).
\]

The S-MEQ method (\ref{iter:teq-1}) developed in Section 3 neglected the term $r_k$.
Since $r_k=O(\|x_{k+1}-x_k\|)$, it might be  important for Newton's method
to be quadratically convergent. As a result, the method (\ref{iter:teq-1}) may not be
a good approximation to Newton's method.
In this section, we consider to improve the S-MEQ method by using more information of $r_k$.

Without loss of generality, we suppose that $\cal M$ is semi-symmetric. So we have
\[
({\cal M}x^{m-1})'=(m-1){\cal M}x^{m-2}.
\]
Denote
\[
r(x)= \frac 1{m-1} ({\cal M}x^{m-1}-(m-1)Mx^{[m-1]}).
\]
We have
\[
r'(x)= {\cal M}x^{m-2}-(m-1)M\text{diag}(x^{[m-2]})
\]
and
\begin{eqnarray*}
r_k &=& {\cal M}x^{m-2}_k (x_{k+1}-x_k)- M (x^{[m-1]}_{k+1}-x_k^{[m-1]})\\
&=&  {\cal M}x^{m-2}_k (x_{k+1}-x_k)- (m-1)M\text{diag}(x_k^{[m-2]})(x_{k+1}-x_k)+o(\|x_{k+1}-x_k\|)\\
    &=& r'(x_k) (x_{k+1}-x_k)+o(\|x_{k+1}-x_k\|)\\
    &=& r(x_{k+1})-r(x_k) +o(\|x_{k+1}-x_k\|)\\
    &=& \frac 1{m-1} (F(x_{k+1})-F(x_k)) - M (x_{k+1}^{[m-1]}-x_k^{[m-1]}) +o(\|x_{k+1}-x_k\|)\\
    &\stackrel\triangle {=} & \frac 1{m-1} y_k - M s_k+o(\|x_{k+1}-x_k\|),
\end{eqnarray*}
where
\[
y_k=F(x_{k+1})-F(x_k),\quad s_k= x_{k+1}^{[m-1]}-x_k^{[m-1]}.
\]
Consequently, it follows from Newton's method that
\begin{eqnarray*}
0 &=& \frac {1}{m-1} F(x_k) +{\cal M}x_k^{m-2}(x_{k+1}-x_k)\\
    &=& \frac {1}{m-1} F(x_k)
        + (M x_{k+1} ^{[m-1]} - Mx_k^{[m-1]}) +
        ( {\cal M}x^{m-2}_k (x_{k+1}-x_k)- M (x^{[m-1]}_{k+1}-x_k^{[m-1]}))\\
    &=& \frac {1}{m-1} F(x_k)
        + (M x_{k+1} ^{[m-1]} - Mx_k^{[m-1]}) + r'(x_k)(x_{k+1}-x_k) +o(\|x_{k+1}-x_k\|)\\
    &=& \frac {1}{m-1} F(x_k)
        + (M x_{k+1} ^{[m-1]} - Mx_k^{[m-1]}) + r(x_{k+1})-r(x_k) +o(\|x_{k+1}-x_k\|),\quad k=0,1,2,\ldots.
\end{eqnarray*}
A reasonable approximation to Newton's method is to let $x_{k+1}$ satisfy
\[
\frac {1}{m-1} F(x_k)
        + (M x_{k+1} ^{[m-1]} - Mx_k^{[m-1]}) + r(x_{k+1})-r(x_k)=0.
\]
However, the point $x_{k+1}$ is not known in advance. So, we consider to
use $\epsilon_k\stackrel\triangle {=}r(x_k)-r(x_{k-1})$ instead of the term $r(x_{k+1})-r(x_k)$
in the last equation. This results in the following iteration:
\begin{equation}\label{iter:improve-1}
M x_{k+1} ^{[m-1]} - Mx_k^{[m-1]}+\delta_k=0,\quad k=0,1,2,\ldots,
\end{equation}
where
\[
\delta_k=\frac {1}{m-1} F(x_k) +\epsilon_k.
\]
More generally, we propose the following iterative scheme:
\begin{equation}\label{iter:improve}
M x_{k+1} ^{[m-1]} - Mx_k^{[m-1]}+\alpha_k F(x_k) +\epsilon_k=0,\quad k=0,1,2,\ldots,
\end{equation}
%
where $\alpha_k\in (0,1]$ and
$\epsilon_k$ is chosen in the way that  $\epsilon_0=0$ and for  $k\ge 1$,
$\epsilon_k = r(x_k)-r(x_{k-1})$ .
However, such a simple choice rule for $\epsilon_k$
could not guarantee the generated sequence $\{x_k\}$ contained in $\cal S$.
As a result, the convergence of the related method is doubtful.
In what follows,
we give some other reasonable choice for $\epsilon_k$ to ensure the monotone convergence
of $\{x_k\}$ to a nonnegative solutio of (\ref{eqn:tensor}).

In order for the method to be monotonically convergent, we need the requirement
$\alpha_k F(x_k)+\epsilon_k\le 0$, $\forall k$  to ensure the monotone property of $\{x_k\}$.
It is satisfied if we let $\epsilon_k$ satisfy
\begin{equation}\label{ep-1}
\epsilon_k\le -\alpha_k F(x_k)\stackrel\triangle {=}\epsilon_k^+.
\end{equation}
On the other hand, we also need the condition
$F(x_k)\le 0$, $\forall k$ to ensure $\{x_k\}\subset {\cal S}$.
This motives us to determine $\epsilon_k$ in the following way.
At iteration $k$, we first let
$\epsilon_k=\min\{\epsilon_k^+,  r(x_k)-r(x_{k-1})\}$ and then solve the system
of linear equations (\ref{iter:improve}) to get a $\bar x_{k+1}$. If $F(\bar x_{k+1})\le 0$,
then we let $x_{k+1}=\bar x_{k+1}$. Otherwise, we let $\epsilon_k=0$ and solve (\ref{iter:improve})
to get $x_{k+1}$.

Based on the above arguments, we propose an approximate Newton method for solving
the M-tensor equation (\ref{eqn:tensor}) as follows.

\begin{algorithm}\label{A-Newton}{(\bf Approximate Newton Method)}
\begin{itemize}
\item [] {\bf Initial. } Given positive sequence $\{\alpha_k\}\subset (0,1]$.
Given constant $\eta>0$ and initial point $x_0\in {\cal S}$.
Let  $\epsilon_0=0$ and $k=0$.
\item [] {\bf Step 1. } Stop if $\|F(x_k)\|\le \eta $.
\item [] {\bf Step 2. } Solve the system of linear equation (\ref{iter:improve}) to get $\bar x_{k+1}$.
\item [] {\bf Step 3. } If $F(\bar x_{k+1})\leq 0$, then go to Step 4. Otherwise, let $\epsilon_k=0$. Go to Step 2.
\item [] {\bf Step 4. } Let $x_{k+1}=\bar x_{k+1}$ and $\epsilon_{k+1}=\min\{\epsilon_{k+1}^+,  r(x_{k+1})-r(x_k)\}$.
Let $k=k+1$. Go to Step 1.
\end{itemize}
\end{algorithm}

{\bf Remark.} Notice that if $\epsilon_k=0$, then the method reduces to the S-MEQM. Consequently, if for some
$k$, $F(\bar x_{k+1})> 0$, then the iterate $x_{k+1}$ is generated by the S-MEQM, which guarantees
$F(x_{k+1})\le 0$. In other words, at each iteration, the circle between Steps 2 and 3 is no more than once.

It is easy to show by induction that the sequence $\{x_k\}$ generated by Algorithm \ref{A-Newton}
satisfies $x_{k+1}\ge x_k$ and $F(x_k)\le 0$, $\forall k$.
Similar to the proof of Theorem \ref{th:conv-1},
we have the following result.

\begin{theorem}\label{th:conv-improve}
Suppose that ${\cal M}$ is a strongly M-tensor.
Then the sequence $\{x_k\}$ be generated by Algorithm \ref{A-Newton} satisfies
$\{x_k\}\subset {\cal S}$ and
\[
x_{k+1}\ge x_k, \quad  \forall k\ge 0.
\]
Moreover, it converges to a nonnegative solution to the tensor equation (\ref{eqn:tensor}).
\end{theorem}

\section{Numerical Results}
\setcounter{equation}{0}

In this section, we do numerical experiments to test the effectiveness of the proposed methods.
We implemented our methods in Matlab R2015b and ran the codes on a laptop computer with 2.40 GHz CPU and 4.00 GB RAM.
We used a tensor toolbox \cite{Bader-Kolda-15} to proceed tensor computation.

While do numerical experiments, similar to \cite{Han-17,He-Ling-Qi-Zhou-18}, we solved the tensor equation
\[
\hat{F}(x)=\hat{\cal {M}}x^{m-1}-\hat{b}=0
\]
instead of the tensor equation (\ref{eqn:tensor}),
where $\hat{\cal{M}}:=\cal{M}/\omega$ and $\hat{b}:=b/\omega$ with $\omega$ is the largest value among
the absolute values of components of $\cal{M}$ and $b$.
We set the initial point to be $x_0=(0,0,\ldots,0)^T$,  and the elements of  ${\emph{b}}$ be uniformly
distributed in $(0,1)$ except for the Problem 3. The stopping criterion  is set to
\[
\|\hat{F}(x_k)\|\leq 10^{-8}.
\]
or the number of iteration reaches to 3000. In all cases, we take the parameter $\alpha_k=\alpha$ be
constant.

The test problems are from \cite{Ding-Wei-16,Li-Xie-Xu-17,Xie-Jin-Wei-2017}.

{\bf Problem 1} We solve tensor equation (\ref{eqn:tensor}) where ${\cal M}$ is a 4-th order symmetric strong M-tensor in the form ${\cal M}=s{\cal I}-{\cal B}$, where tensor ${\cal B}$ is symmetric
whose entries are uniformly distributed in $(0,1)$, and
\[
s=(1+0.01)\cdot\max_{i=1,2,\ldots,n}({\cal B}{\bf e}^{3})_i,
\]
where ${\bf e}=(1,1,\ldots,1)^T$.

{\bf Problem 2.} We solve tensor equation (\ref{eqn:tensor}) where $\cal M$ is a 4-th order symmetric strong M-tensor in the form ${\cal M}=s{\cal I}-{\cal B}$, and tensor ${\cal B}$ is a nonnegative tensor with
\[
b_{i_1i_2i_3i_4}=|\mathrm{sin}(i_1+i_2+i_3+i_4)|,
\]
and $s=n^3$.

{\bf Problem 3.} We consider the ordinary differential equations
\[
\frac{d^2x(t)}{dt^2}=-\frac{f(t)}{x(t)^2}  \quad\mathrm {in \,(0,1),}
\]
with Dirichlet's boundary conditions
\[
x(0)=c_0,\,\,\,x(1)=c_1.
\]

This equation is the one-dimensional situation of the boundary value problem, which can be
used to describe a particle's movement under the gravitation
\[
\frac{d^2x(t)}{dt^2}=-\frac{GM}{x(t)^2}
\]
where $G\approx6.67\times10^{-11}Nm^2/kg^2$ is the gravitational constant and $M\approx5.98\times10^{24}kg$ is the mass of the earth.\\

After the discretization, we get a system of polynomial equations

\[\left \{ \begin{array}{cll}
& x_1^3=c_0^3,\\
& 2x_i^3-x_i^2x_{i-1}-x_i^2x_{i+1}=\frac{GM}{(n-1)^2},\,\,i=2,3,\ldots,n-1,\\
& x_n^3=c_1^3.
\end{array}\right.
\]

It is easy to see that the last system of equations can be written as a tensor equation
\[
{\cal A}x^3=b,
\]
where the coefficient tensor $\cal A$ is a 4th-order tensor with elements
\[\left \{ \begin{array}{cll}
& a_{1111}=a_{nnnn}=1,\\
& a_{iiii}=2,\,\,i=2,3,\ldots,n-1,\\
& a_{i(i-1)ii}=a_{ii(i-1)i}=a_{iii(i-1)}=-1/3,\,\,i=2,3,\ldots,n-1\\
& a_{i(i+1)ii}=a_{ii(i+1)i}=a_{iii(i+1)}=-1/3,\,\,i=2,3,\ldots,n-1,
\end{array}\right.
\]
and the right-hand side is a vector with elements
\[\left \{ \begin{array}{cll}
& b_1=c_0^3,\\
& b_i=\frac{GM}{(n-1)^2},\,\,i=2,3,\ldots,n-1\\
& b_n=c_1^3.
\end{array}\right.
\]
We set $c_0=c_1=6.37\times 10^6$.

{\bf Problem 4.} We solve tensor equation (\ref{eqn:tensor}) where $\cal M$ is a 4-th order non-symmetric strong M-tensor in the form ${\cal M}=s{\cal I}-{\cal B}$, and tensor ${\cal B}$ is nonnegative
tensor whose entries are uniformly distributed in $(0,1)$.
The parameter $s$ is set to
\[
s=(1+0.01)\cdot\max_{i=1,2,\ldots,n}({\cal B}{\bf e}^{3})_i,
\]
where ${\bf e}=(1,1,\ldots,1)^T$.

We first tested the performance of the sequential M matrix equation method (\ref{iter:teq-1}), abbreviated as S-MEQM,
with different values of $\alpha$ on the following
Problem 1 with different dimensions $n=10$, $20$, $30$, $40$, $50$.
 For each $\alpha$,
we tested  the method S-MEQM on 100 problems with different sizes $n=10,20,30,40$ and $50$.
The results are listed in Table \ref{alpha-S-MEQM} where the columns `Iter' and `Time' stand for the
total number of iterations
and the computational time (in second) used for the method.
The results with $\alpha\in (0,0.5)$ were not listed in the table because those results are not as
good as the method with $\alpha \in (0.5, 1]$.
We also tested the S-MEQM with the value $\alpha>1$.
 Although we could not establish the convergence of the method  S-MEQM
with $\alpha>1$, the numerical results in Table\ref{alpha-S-MEQM}  seem to show that the method still works for all $\alpha \in (1,2)$.
Moreover, the best value of $\alpha$ seems around $1.9$.

We then tested the approximate Newton method Algorithm \ref{A-Newton}
which we  abbreviated as A-Newton. Table \ref{A-Newton-1} lists the performance of the method A-Newton with
different values of $\alpha$ on Problem 1 with different dimensions. For each $\alpha$,
we also tested  the method A-Newton on 100 problems with different sizes $n=10,20,30,40$ and $50$.
We see from the table that the best parameter $\alpha$ seems near $1$.

The results in Tables \ref{alpha-S-MEQM} and \ref{A-Newton-1} show that the A-Newton method performed
much better then S-MEQM
if we choose the parameter $\alpha$ appropriately.

{
\begin{table}
\footnotesize
\centering
\begin{tabular}{c|cc|cc|cc|cc|cc}
\hline\hline
$n$ & \multicolumn{2}{c|}{10} & \multicolumn{2}{c|}{20} & \multicolumn{2}{c|}{30} & \multicolumn{2}{c|}{40} & \multicolumn{2}{c}{50}\\\hline
$\alpha$ & Iter & Time & Iter & Time & Iter & Time & Iter & Time & Iter & Time\\\hline
0.50 	&	1000.5 	&	0.1613 	&	1189.2 	&	0.2183 	&	1217.4 	&	1.1088 	&	1192.2 	&	3.1413 	&	1139.9 	&	8.4224\\
0.60 	&	833.6 	&	0.1149 	&	991.8 	&	0.1819 	&	1010.7 	&	0.9063 	&	988.5 	&	2.5955 	&	950.4 	&	6.9442\\
0.70 	&	694.8 	&	0.0954 	&	854.2 	&	0.1564 	&	871.5 	&	0.7631 	&	853.1 	&	2.2383 	&	813.4 	&	5.9696\\
0.80 	&	604.6 	&	0.0827 	&	768.5 	&	0.1415 	&	759.7 	&	0.6896 	&	745.3 	&	1.9605 	&	708.9 	&	5.2117\\
0.90 	&	557.5 	&	0.0761 	&	665.6 	&	0.1221 	&	668.0 	&	0.6209 	&	657.0 	&	1.7505 	&	632.3 	&	4.6709\\
1.00 	&	481.6 	&	0.0656 	&	602.1 	&	0.1107 	&	611.0 	&	0.5463 	&	597.6 	&	1.5566 	&	570.0 	&	4.1866\\ \hline\hline
1.10 	&	447.6 	&	0.0606 	&	539.6 	&	0.0980 	&	553.4 	&	0.5056 	&	536.0 	&	1.4084 	&	517.7 	&	3.8014\\
1.20 	&	414.0 	&	0.0559 	&	501.9 	&	0.0910 	&	509.6 	&	0.4681 	&	494.3 	&	1.2887 	&	474.0 	&	3.5235\\
1.30 	&	376.2 	&	0.0514 	&	465.9 	&	0.0847 	&	466.2 	&	0.4180 	&	455.4 	&	1.1821 	&	438.5 	&	3.1631\\
1.40 	&	345.0 	&	0.0470 	&	425.8 	&	0.0790 	&	433.9 	&	0.3815 	&	419.8 	&	1.1067 	&	405.0 	&	2.9130\\
1.50 	&	332.9 	&	0.0444 	&	393.9 	&	0.0720 	&	404.1 	&	0.3528 	&	393.9 	&	1.0292 	&	378.2 	&	2.7132\\
1.60 	&	305.8 	&	0.0418 	&	366.5 	&	0.0661 	&	378.8 	&	0.3379 	&	369.3 	&	0.9584 	&	354.3 	&	2.5471\\
1.70 	&	289.4 	&	0.0396 	&	349.2 	&	0.0624 	&	355.4 	&	0.3269 	&	349.4 	&	0.9258 	&	332.9 	&	2.4008\\
1.80 	&	288.8 	&	0.0399 	&	325.2 	&	0.0586 	&	336.7 	&	0.3023 	&	327.6 	&	0.8502 	&	315.3 	&	2.2606\\
1.90 	&	253.1 	&	0.0342 	&	314.4 	&	0.0569 	&	319.4 	&	0.2949 	&	310.1 	&	0.8047 	&	298.2 	&	2.1302\\
1.93 	&	264.0 	&	0.0353 	&	309.6 	&	0.0683 	&	317.3 	&	0.2766 	&	304.9 	&	0.8059 	&	294.1 	&	2.1562\\
1.94 	&	265.5 	&	0.0365 	&	302.1 	&	0.0549 	&	312.1 	&	0.2815 	&	303.3 	&	0.7963 	&	290.5 	&	2.1227\\
1.95 	&	301.1 	&	0.0406 	&	301.6 	&	0.0553 	&	309.9 	&	0.2696 	&	303.7 	&	0.7941 	&	289.9 	&	2.1168\\
1.96 	&	411.2 	&	0.0561 	&	304.9 	&	0.0549 	&	311.1 	&	0.2702 	&	300.8 	&	0.7931 	&	289.0 	&	2.1125\\
1.97 	&	736.4 	&	0.1009 	&	355.4 	&	0.0660 	&	322.1 	&	0.2817 	&	307.1 	&	0.8119 	&	292.1 	&	2.1330\\
2.00 	&	3000.0 	&	0.4108 	&	3000.0 	&	0.5674 	&	3000.0 	&	2.6859 	&	3000.0 	&	7.8874 	&	3000.0 	&	21.3840\\ \hline\hline
\end{tabular}\vspace{-3mm}
\caption{\footnotesize Results for S-MEQM with different $\alpha_k$ on Problem 1.}\label{alpha-S-MEQM}
\end{table}
}

{
\begin{table}
\footnotesize
\centering
\begin{tabular}{c|cc|cc|cc|cc|cc}\hline\hline
$n$ & \multicolumn{2}{c|}{ 10} & \multicolumn{2}{c|}{20} & \multicolumn{2}{c|}{30} & \multicolumn{2}{c|}{40}
& \multicolumn{2}{c}{50} \\ \hline
$\alpha_k$ & Iter & Time & Iter & Time & Iter & Time & Iter & Time & Iter & Time \\ \hline
0.10 	&	87.4 	&	0.0455 	&	89.5 	&	0.0536 	&	86.7 	&	0.0959 	&	81.7 	&	0.2099 	&	82.5 	&	0.5859\\
0.20 	&	70.8 	&	0.0401 	&	71.1 	&	0.0399 	&	67.6 	&	0.0677 	&	63.9 	&	0.1643 	&	59.9 	&	0.4431\\
0.30 	&	63.2 	&	0.0349 	&	61.5 	&	0.0213 	&	58.5 	&	0.0514 	&	55.2 	&	0.1424 	&	51.9 	&	0.3717\\
0.40 	&	58.6 	&	0.0314 	&	54.9 	&	0.0155 	&	54.2 	&	0.0482 	&	49.0 	&	0.1270 	&	48.7 	&	0.3488\\
0.50 	&	56.0 	&	0.0287 	&	54.1 	&	0.0149 	&	50.4 	&	0.0455 	&	45.0 	&	0.1166 	&	44.4 	&	0.3189\\
0.60 	&	53.0 	&	0.0250 	&	52.4 	&	0.0151 	&	47.4 	&	0.0488 	&	45.2 	&	0.1157 	&	43.5 	&	0.3132\\
0.70 	&	57.3 	&	0.0321 	&	49.4 	&	0.0139 	&	46.1 	&	0.0425 	&	39.7 	&	0.1028 	&	42.8 	&	0.3083\\
0.80 	&	53.6 	&	0.0309 	&	48.7 	&	0.0123 	&	46.8 	&	0.0417 	&	39.7 	&	0.1030 	&	41.3 	&	0.3211\\
0.90 	&	48.5 	&	0.0174 	&	46.2 	&	0.0114 	&	46.0 	&	0.0407 	&	39.3 	&	0.1018 	&	40.5 	&	0.3165\\
1.00 	&	50.9 	&	0.0128 	&	44.5 	&	0.0101 	&	40.3 	&	0.0358 	&	37.0 	&	0.1016 	&	37.8 	&	0.3294\\
 \hline\hline
\end{tabular}
\caption{\footnotesize Results for A-Newton method with different $\alpha_k$ on Problem 1.}\label{A-Newton-1}
\end{table}

}

In order to compare the performance between the methods S-MEQM and A-Newton, for each problem, we draw
some figures to show the performance of the two methods. In order to ensure the convergence of the method, we select the parameters $\alpha=1$ in both S-MEQM and A-Newton.
Details are given in Figures 1-8. We can clearly see from those figures that for all test problems, the A-Newton
method performed much better than the S-MEQM.

\begin{figure}[H]\label{fig:1}
\centering
\includegraphics[width=0.33\textwidth]{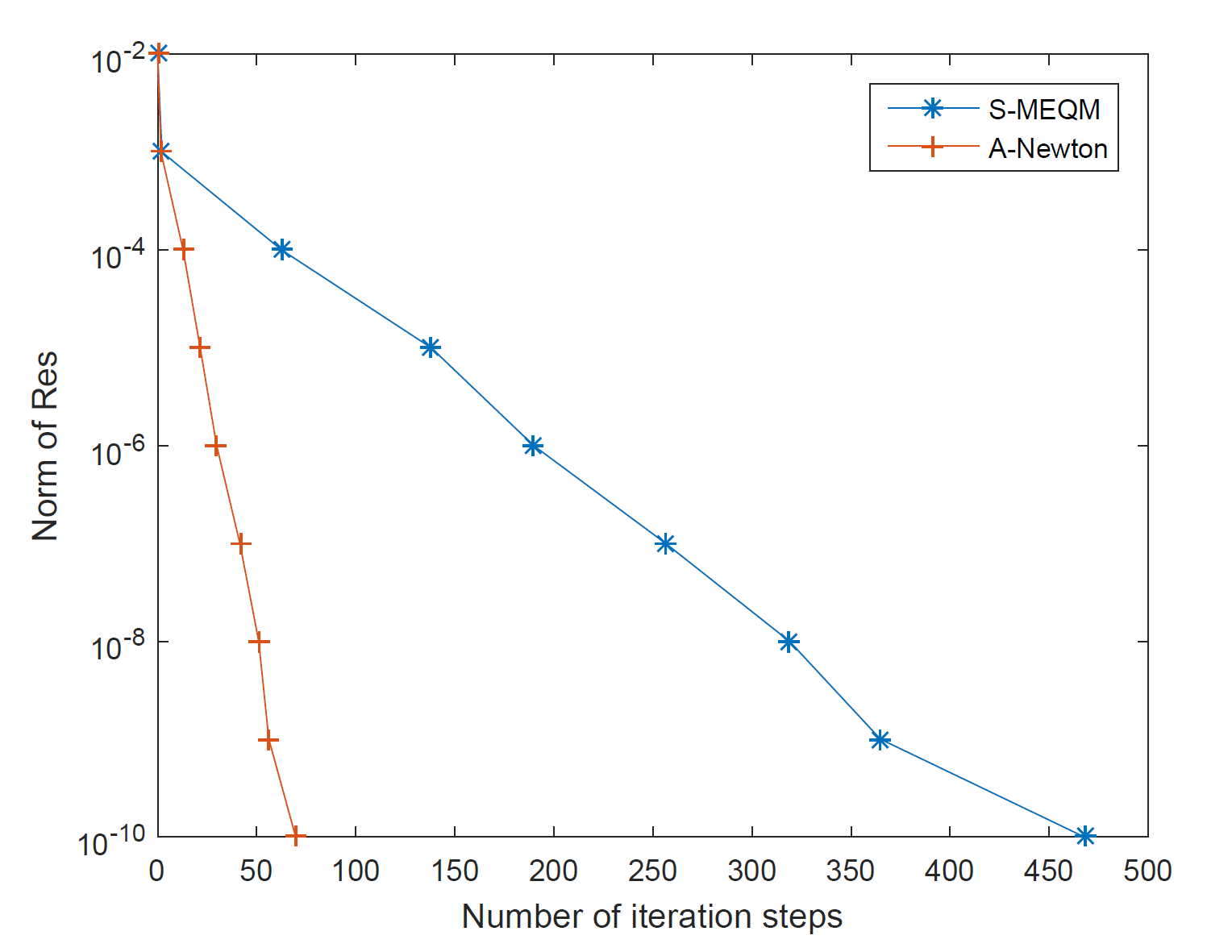}
\includegraphics[width=0.32\textwidth]{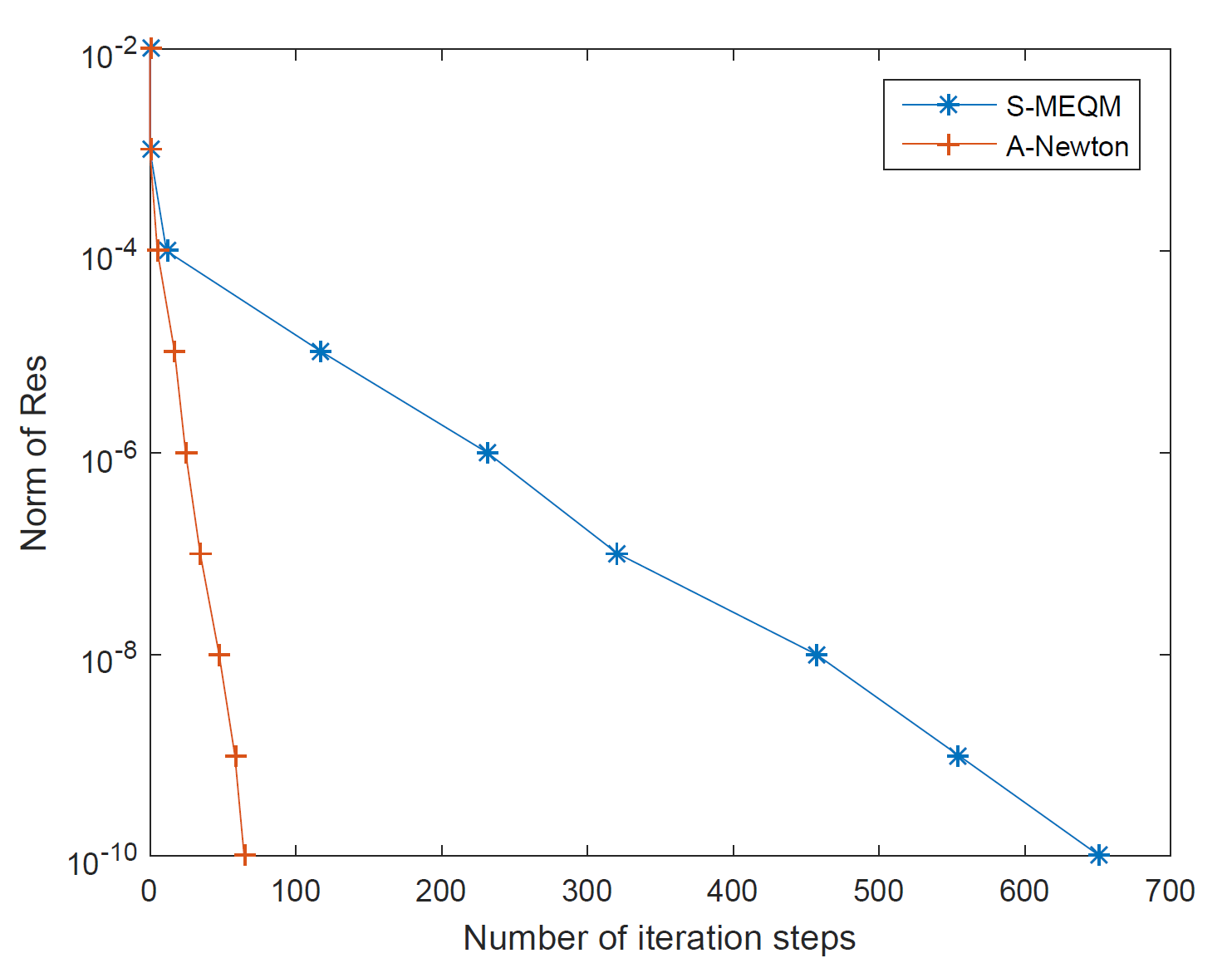}
\includegraphics[width=0.32\textwidth]{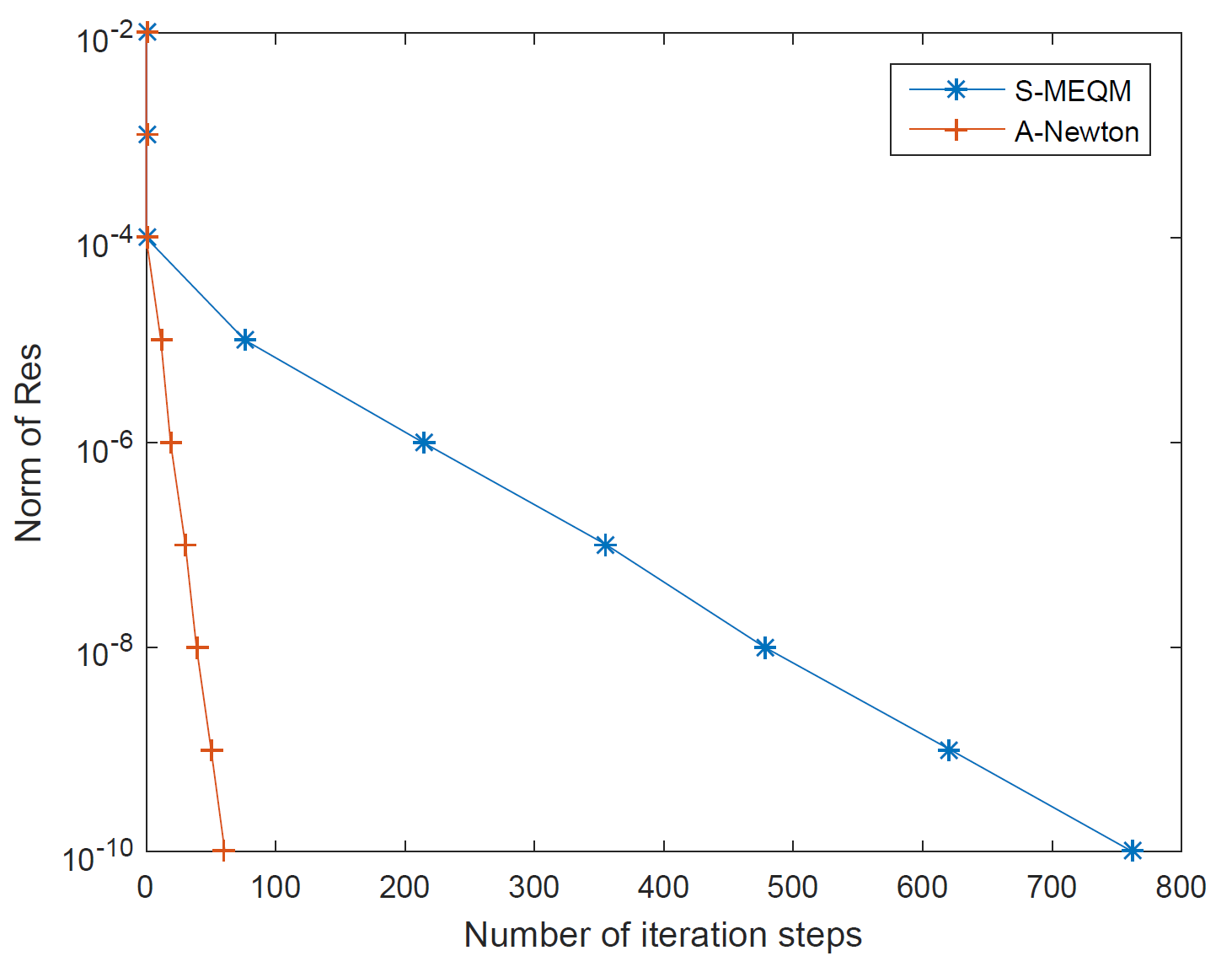}
\caption{\footnotesize Comparison between  S-MEQM and A-Newton on Problem 1: $n=10$ (left), $n=20$ (middle), $n=30$ (right)}

\end{figure}

\begin{figure}[H]\label{fig:2}
\centering
\includegraphics[width=0.34\textwidth]{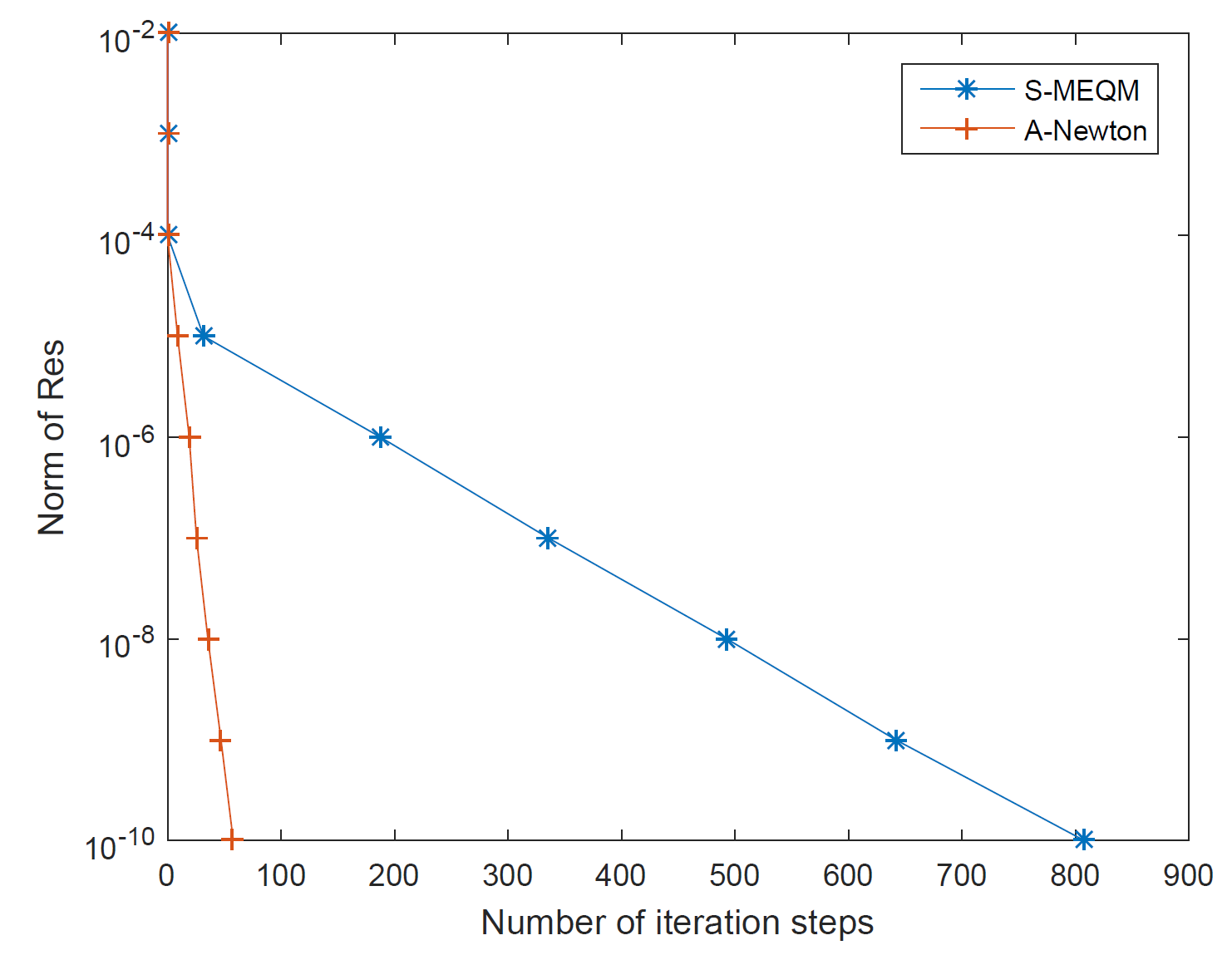}
\includegraphics[width=0.34\textwidth]{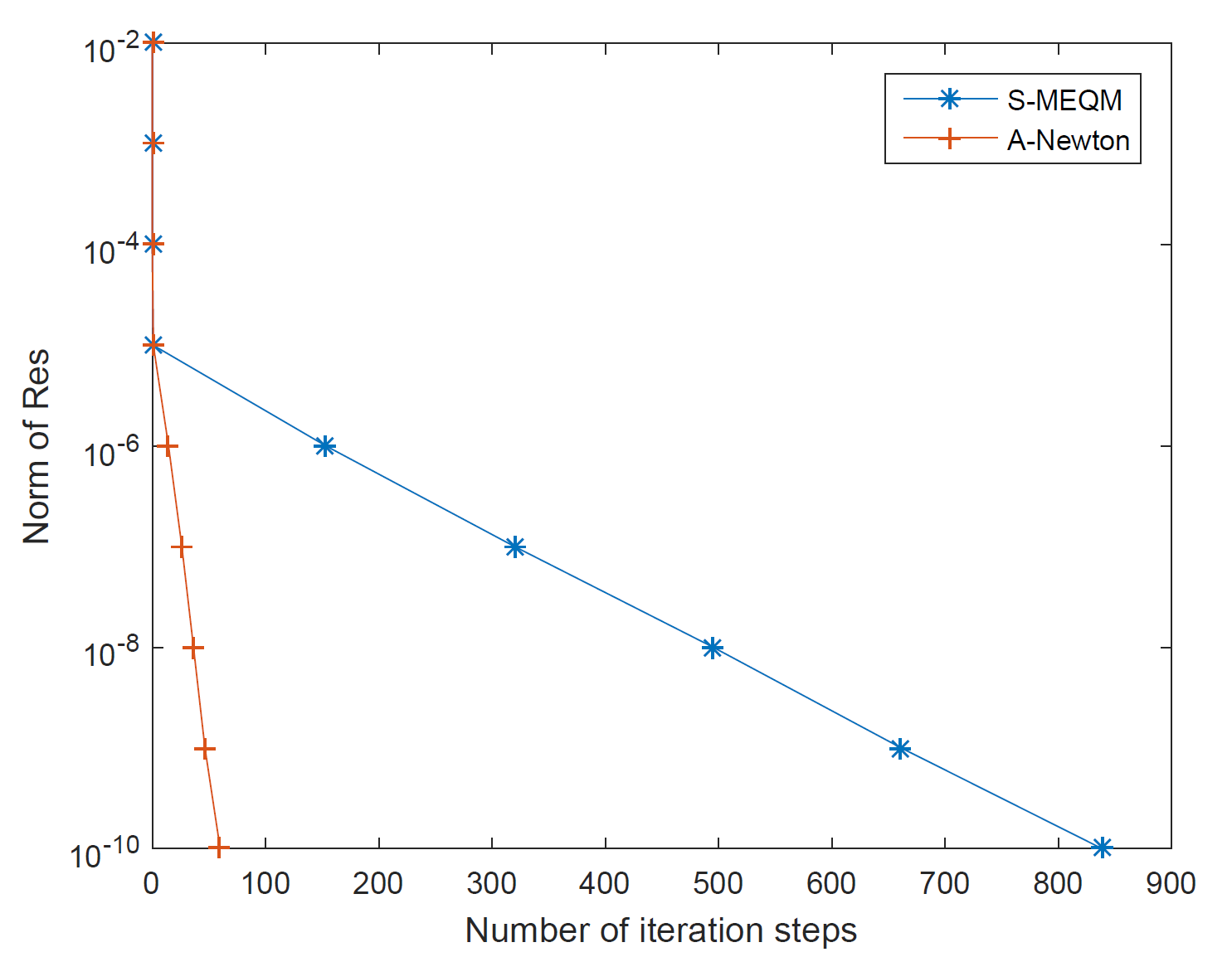}
\caption{\footnotesize Comparison between  S-MEQM and A-Newton on Problem 1: $n=40$ (left), $n=50$ (right)}
\end{figure}

\begin{figure}[H]
\centering
\includegraphics[width=0.32\textwidth]{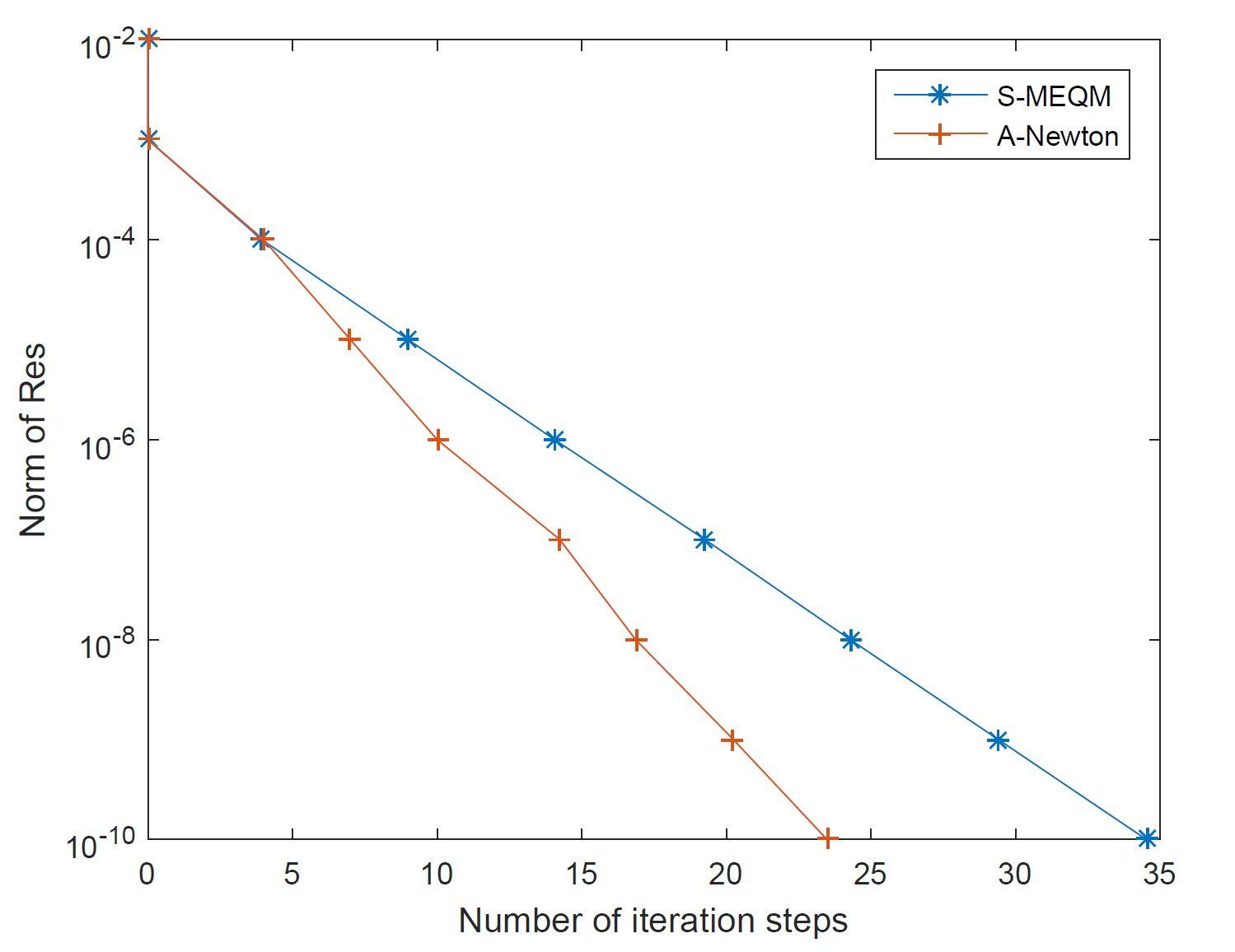}
\includegraphics[width=0.32\textwidth]{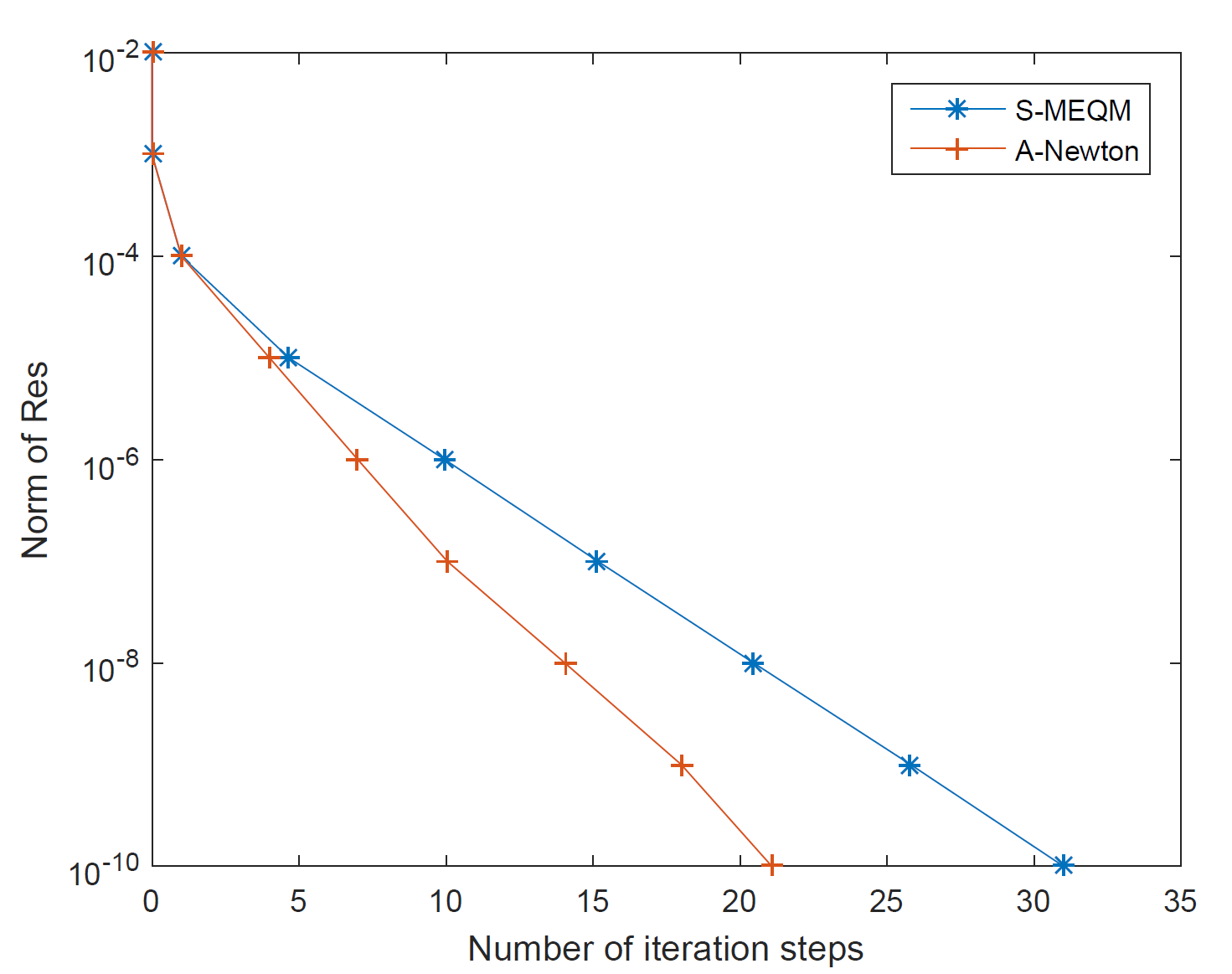}
\includegraphics[width=0.31\textwidth]{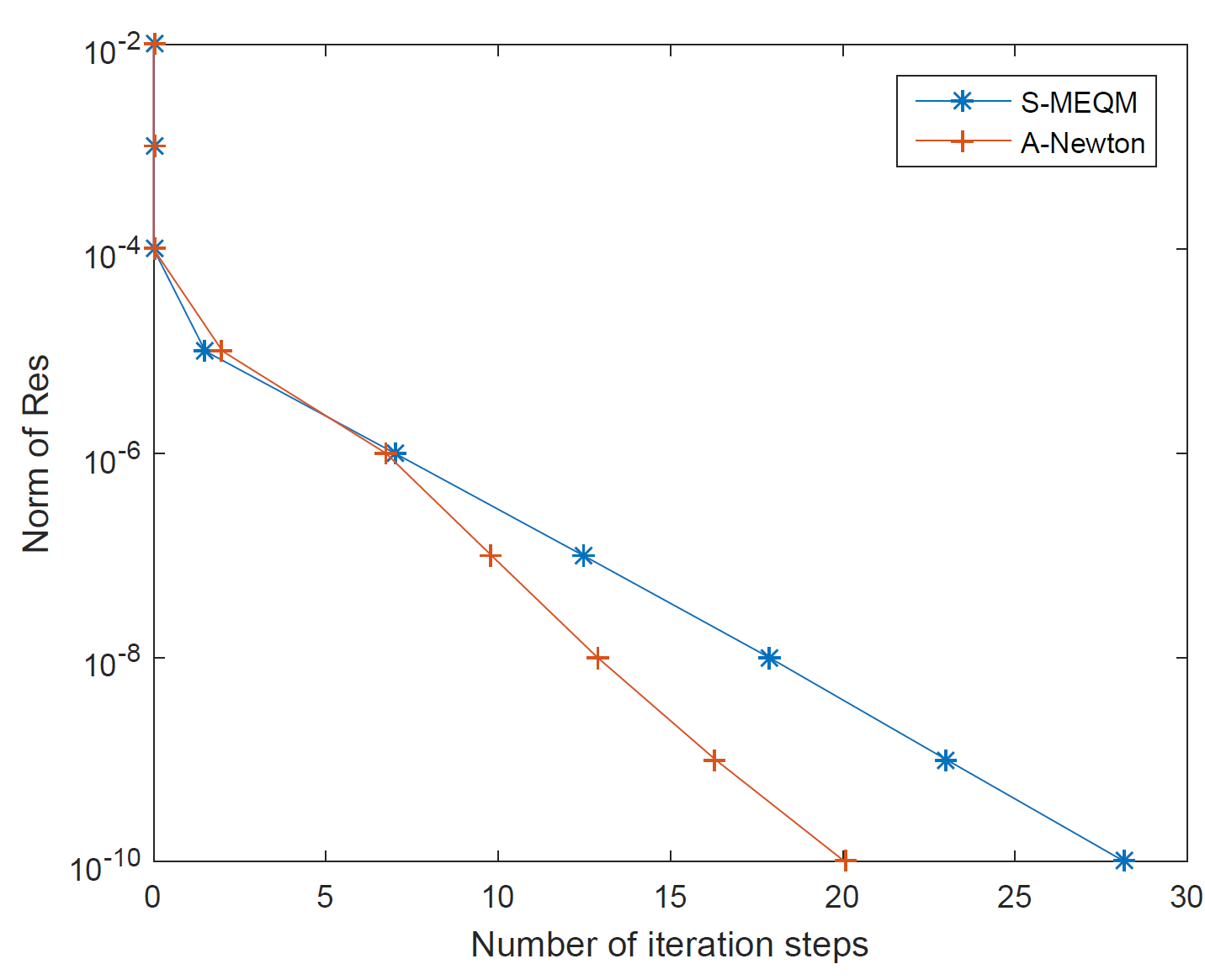}

\caption{\footnotesize Comparison between  S-MEQM and A-Newton on Problem 2: $n=10$ (left), $n=20$ (middle), $n=30$ (right)}

\end{figure}
\begin{figure}[H]
\centering
\includegraphics[width=0.33\textwidth]{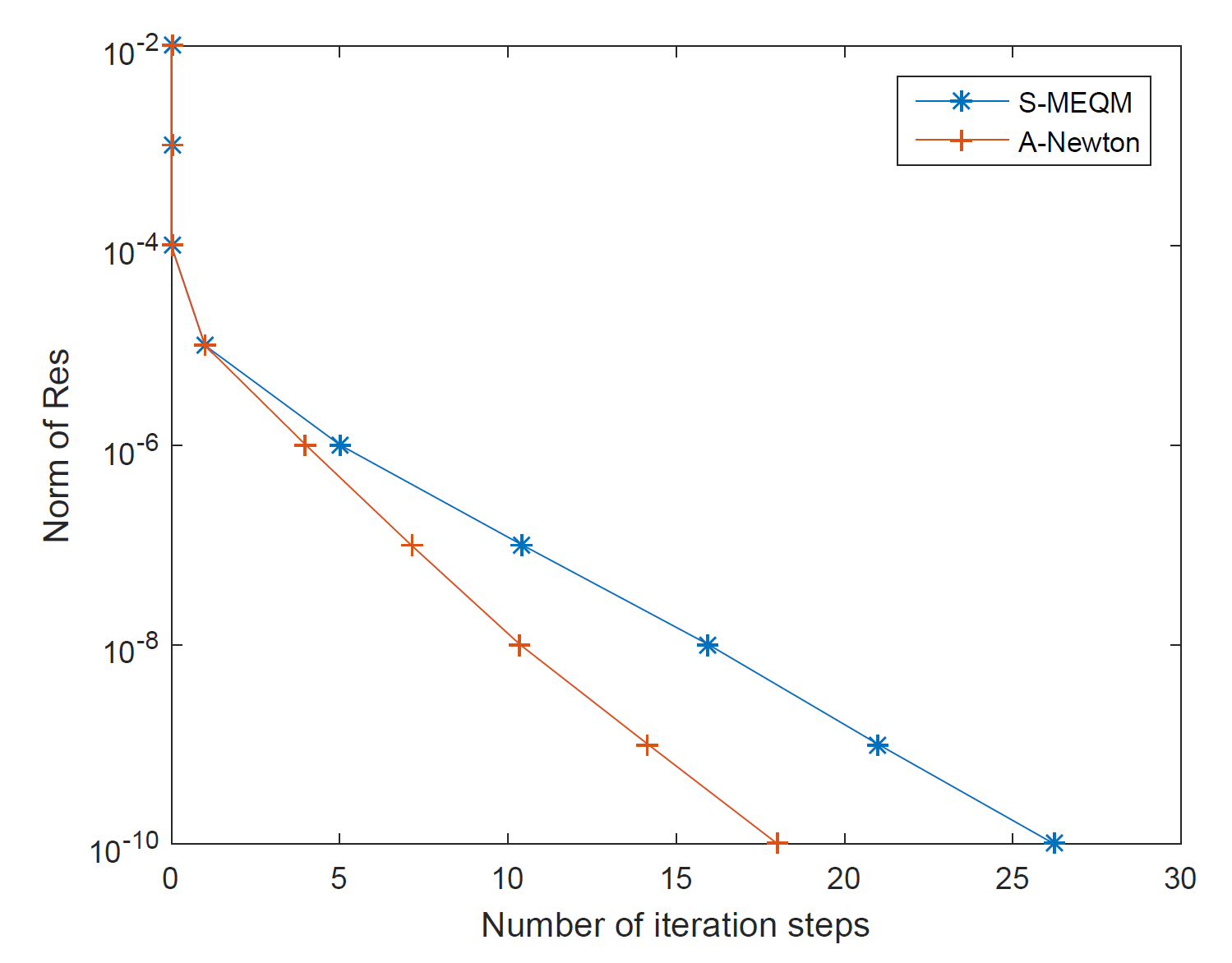}
\includegraphics[width=0.32\textwidth]{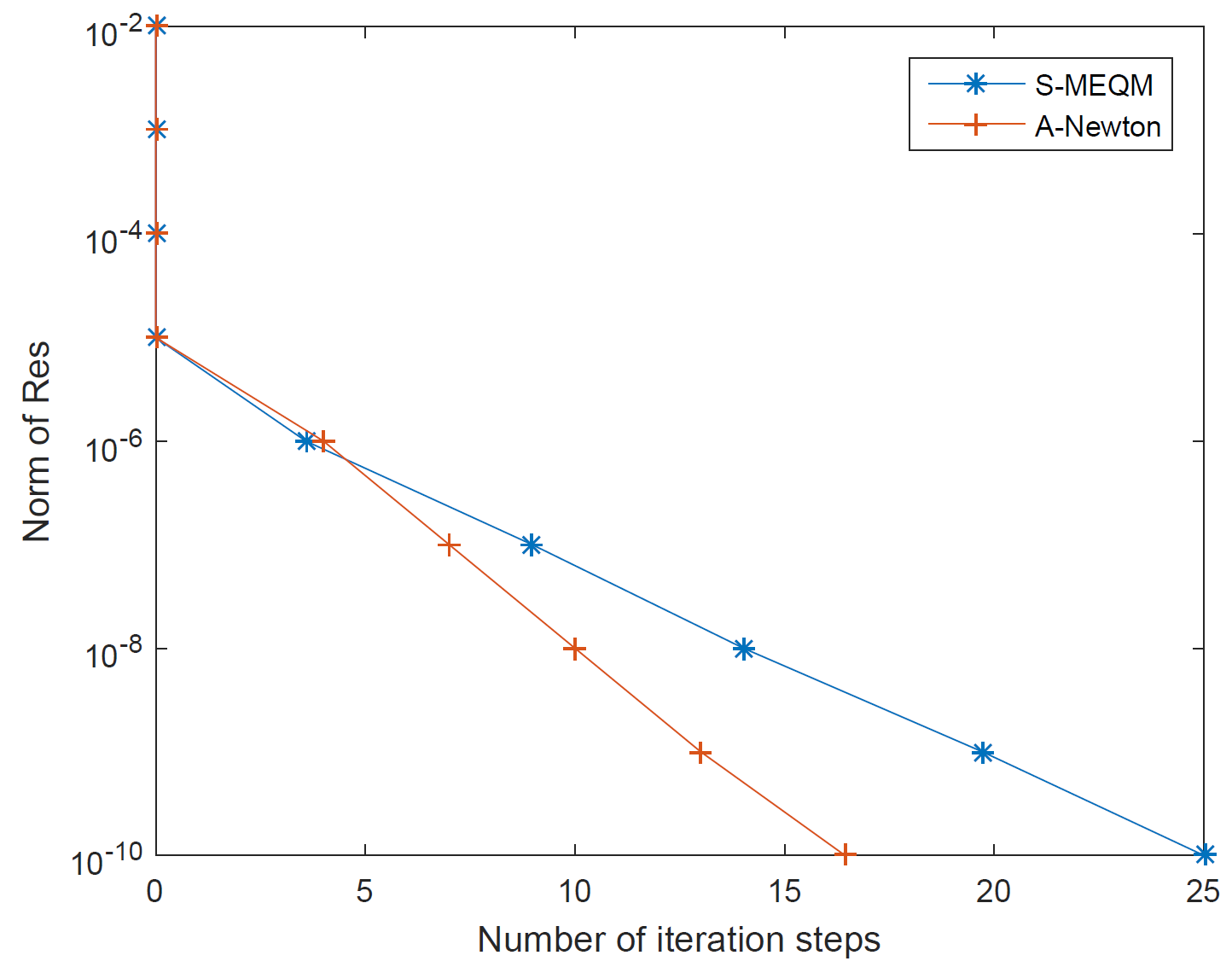}
\caption{\footnotesize Comparison between  S-MEQM and A-Newton on Problem 2: $n=40$ (left), $n=50$ (right)}
\end{figure}

\begin{figure}[H]
\centering
\includegraphics[width=0.32\textwidth]{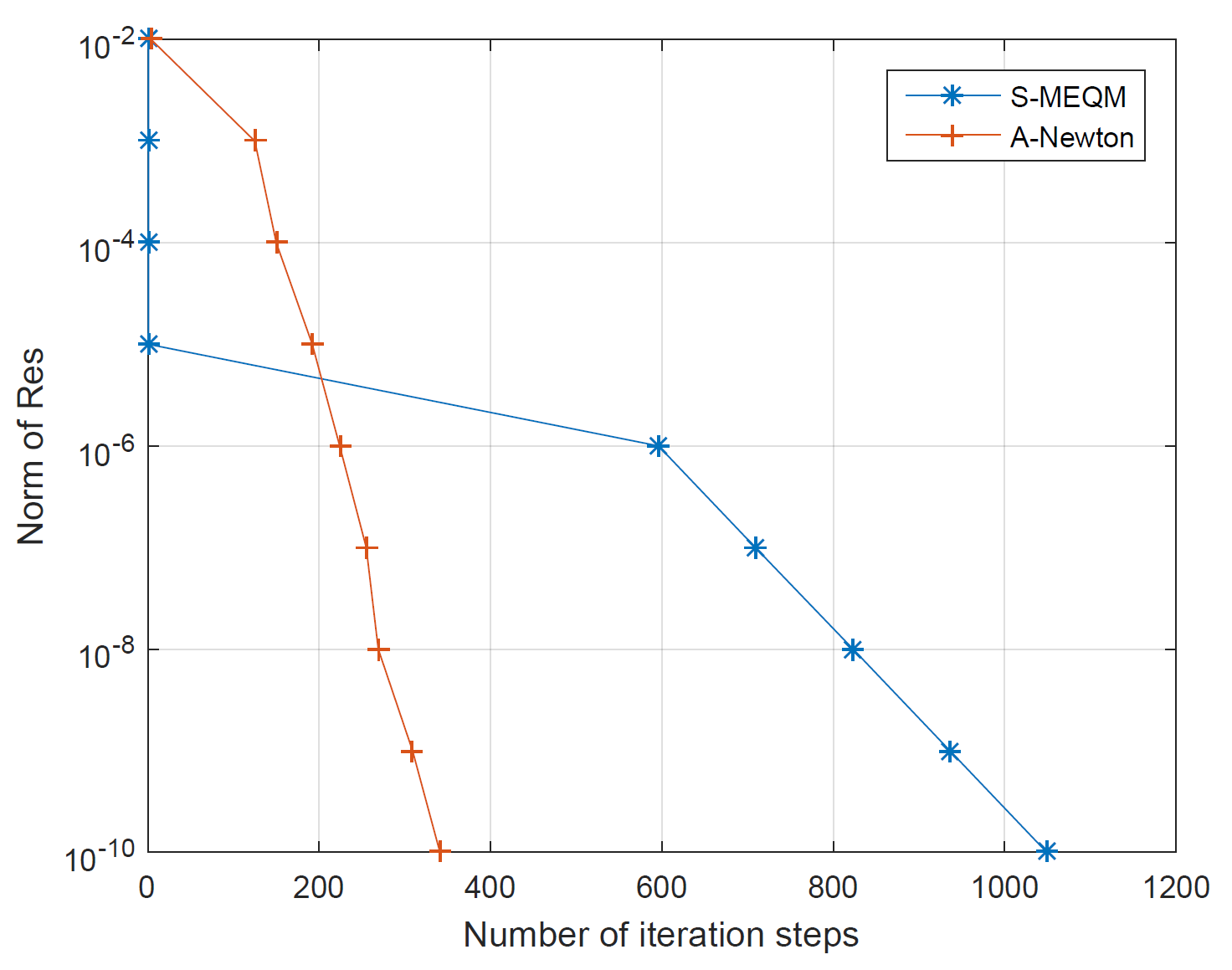}
\includegraphics[width=0.32\textwidth]{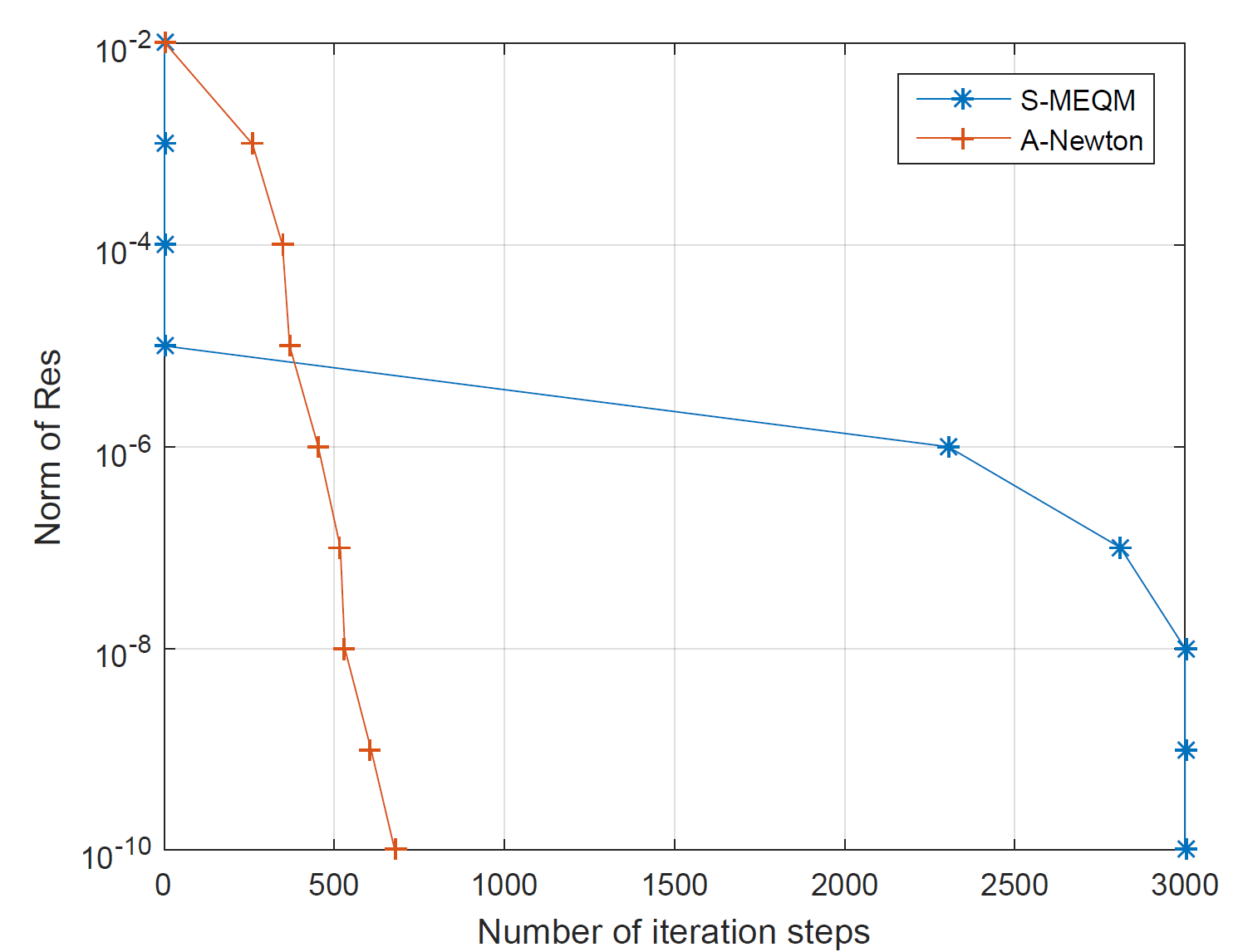}
\includegraphics[width=0.32\textwidth]{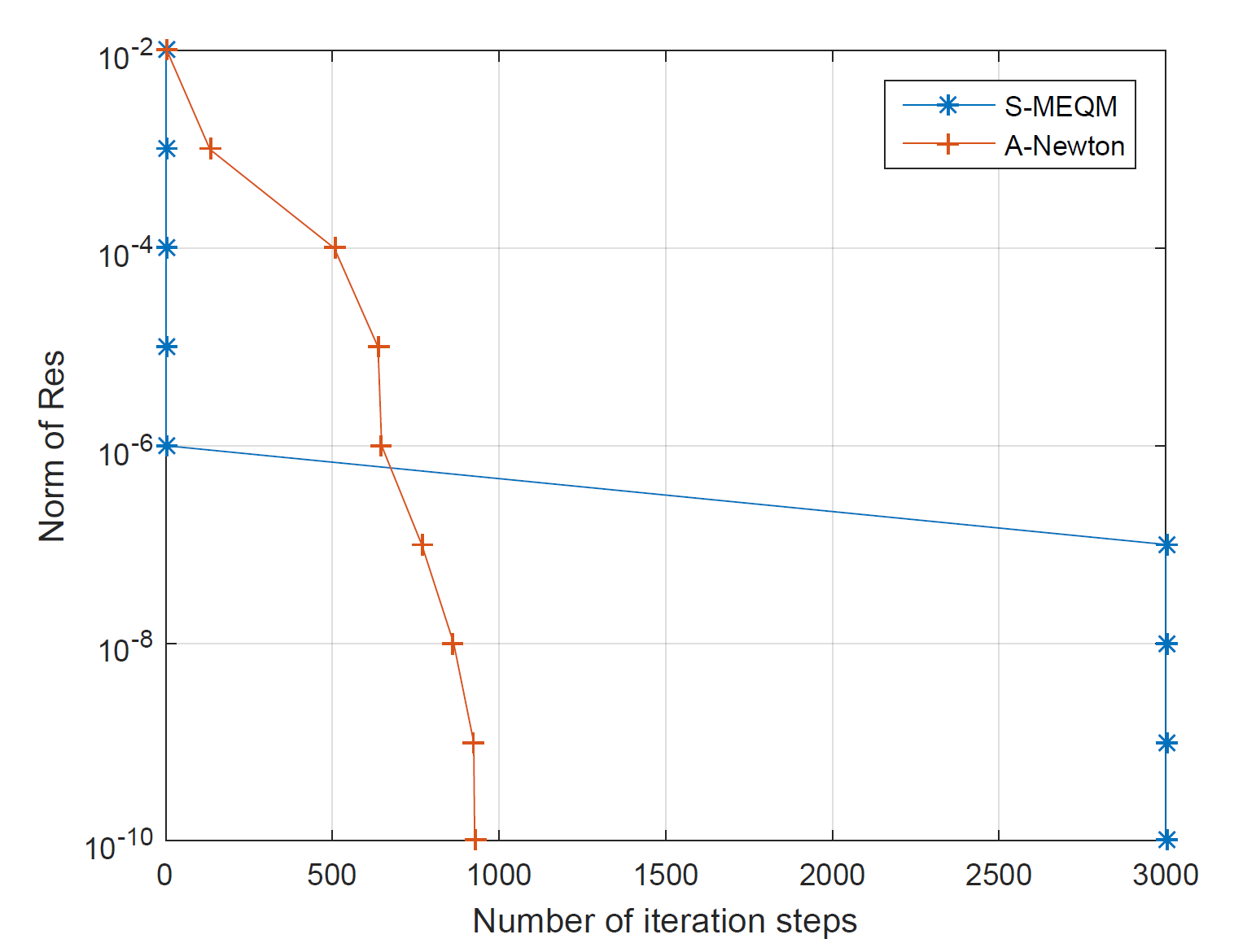}
\caption{\footnotesize Comparison between  S-MEQM and A-Newton on Problem 3: $n=10$ (left), $n=20$ (middle), $n=30$ (right)}

\end{figure}
\begin{figure}[H]
\centering
\includegraphics[width=0.33\textwidth]{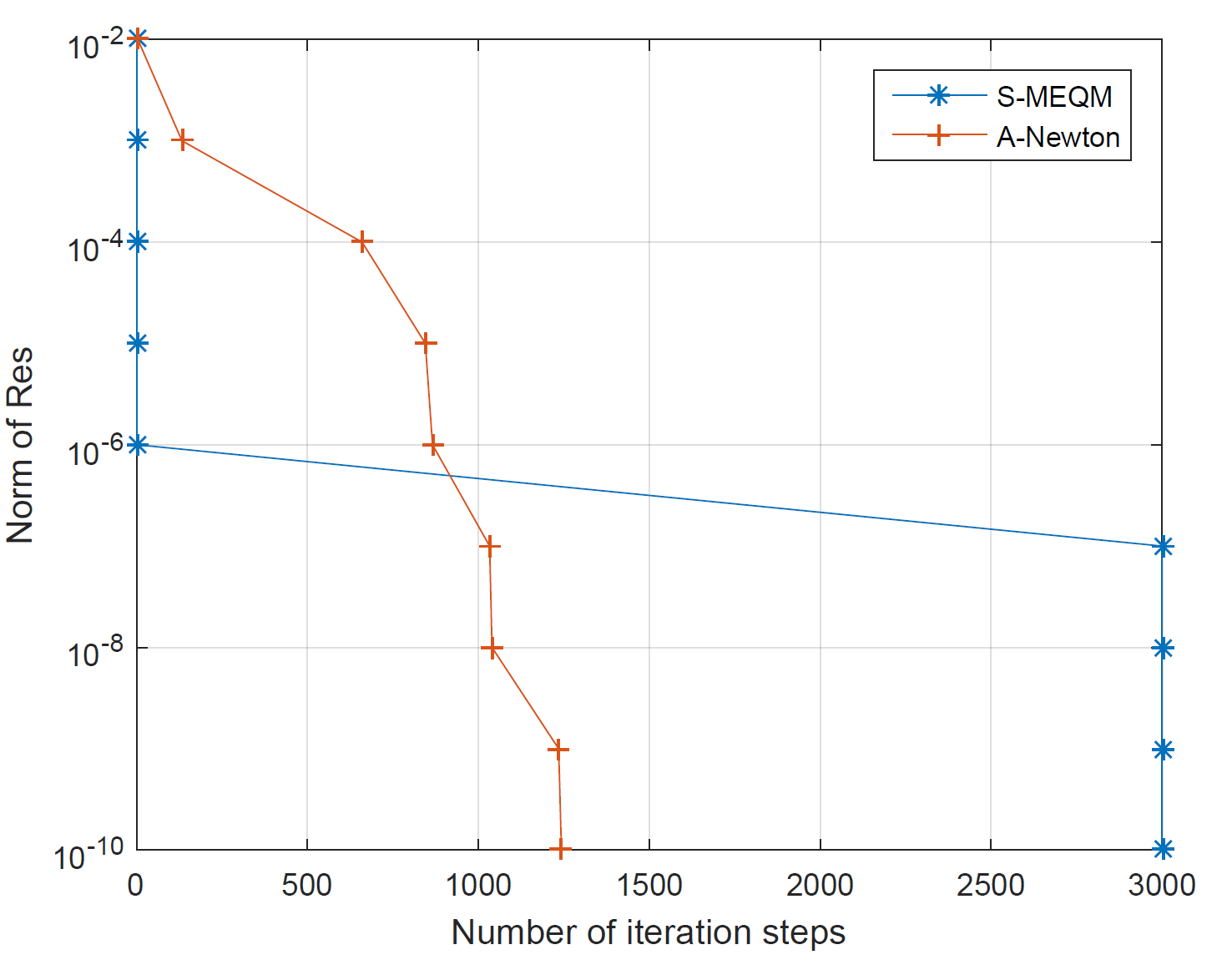}
\includegraphics[width=0.33\textwidth]{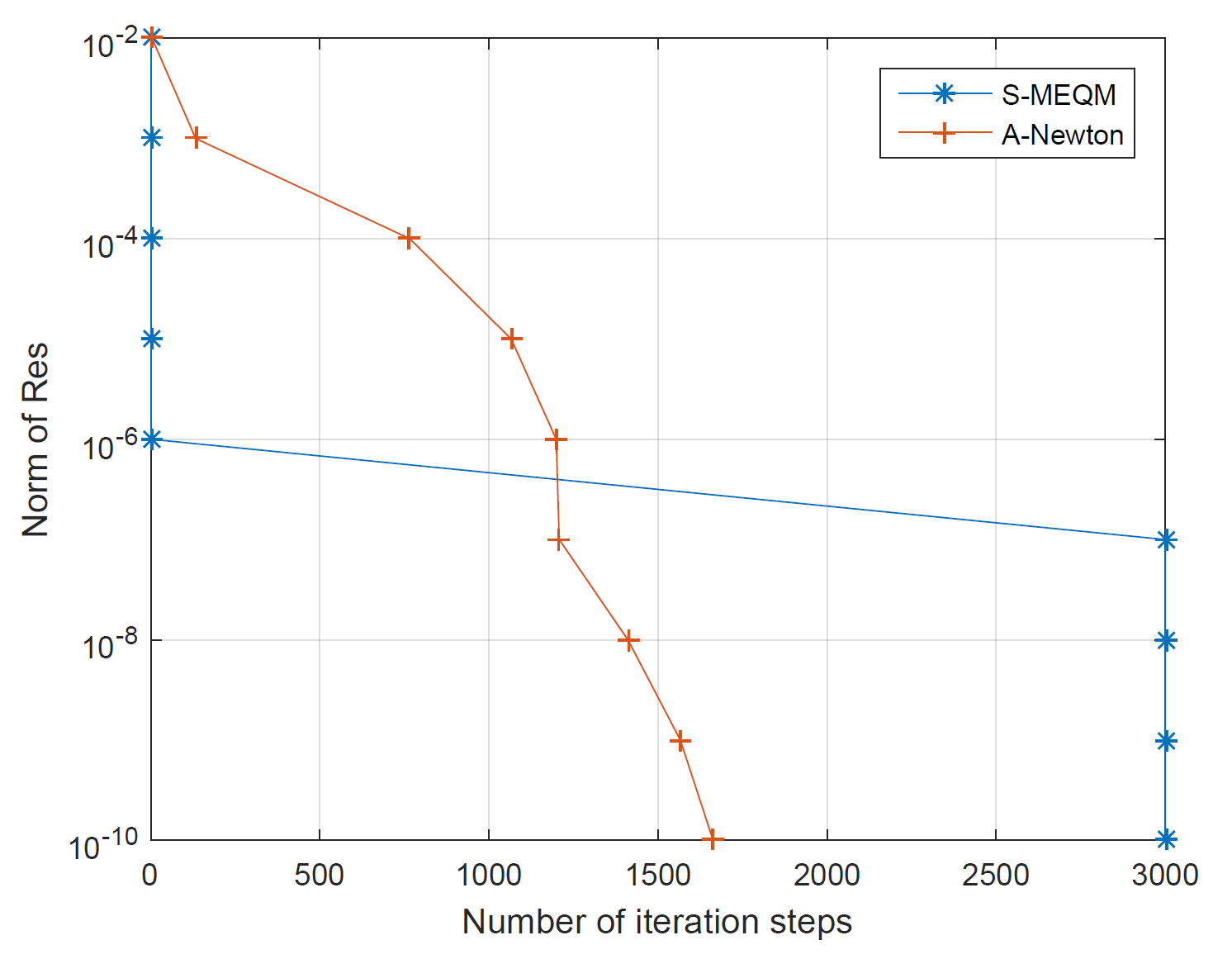}
\caption{\footnotesize Comparison between  S-MEQM and A-Newton on Problem 3: $n=40$ (left), $=50$ (right)}
\end{figure}

\begin{figure}[H]
\centering
\includegraphics[width=0.32\textwidth]{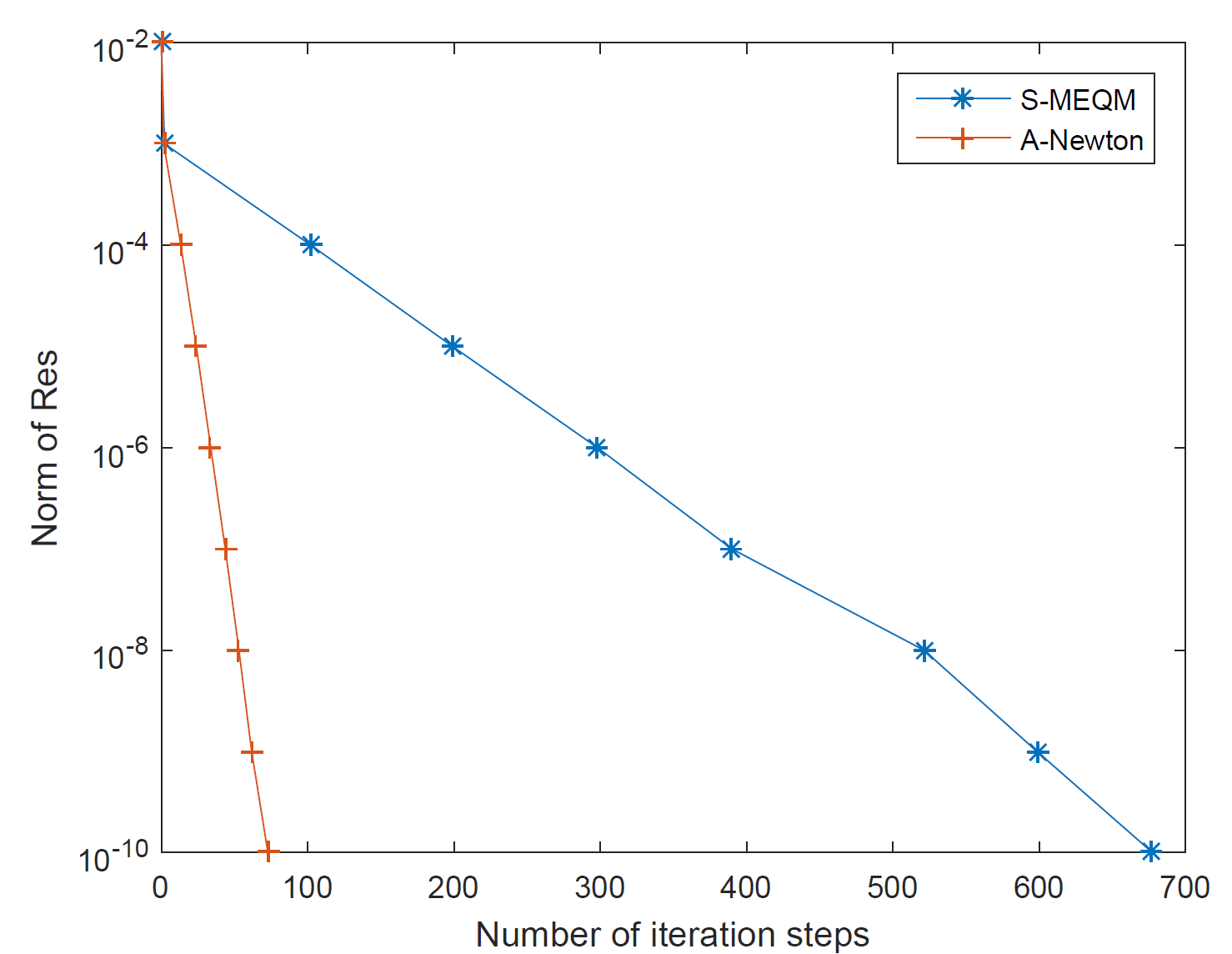}
\includegraphics[width=0.32\textwidth]{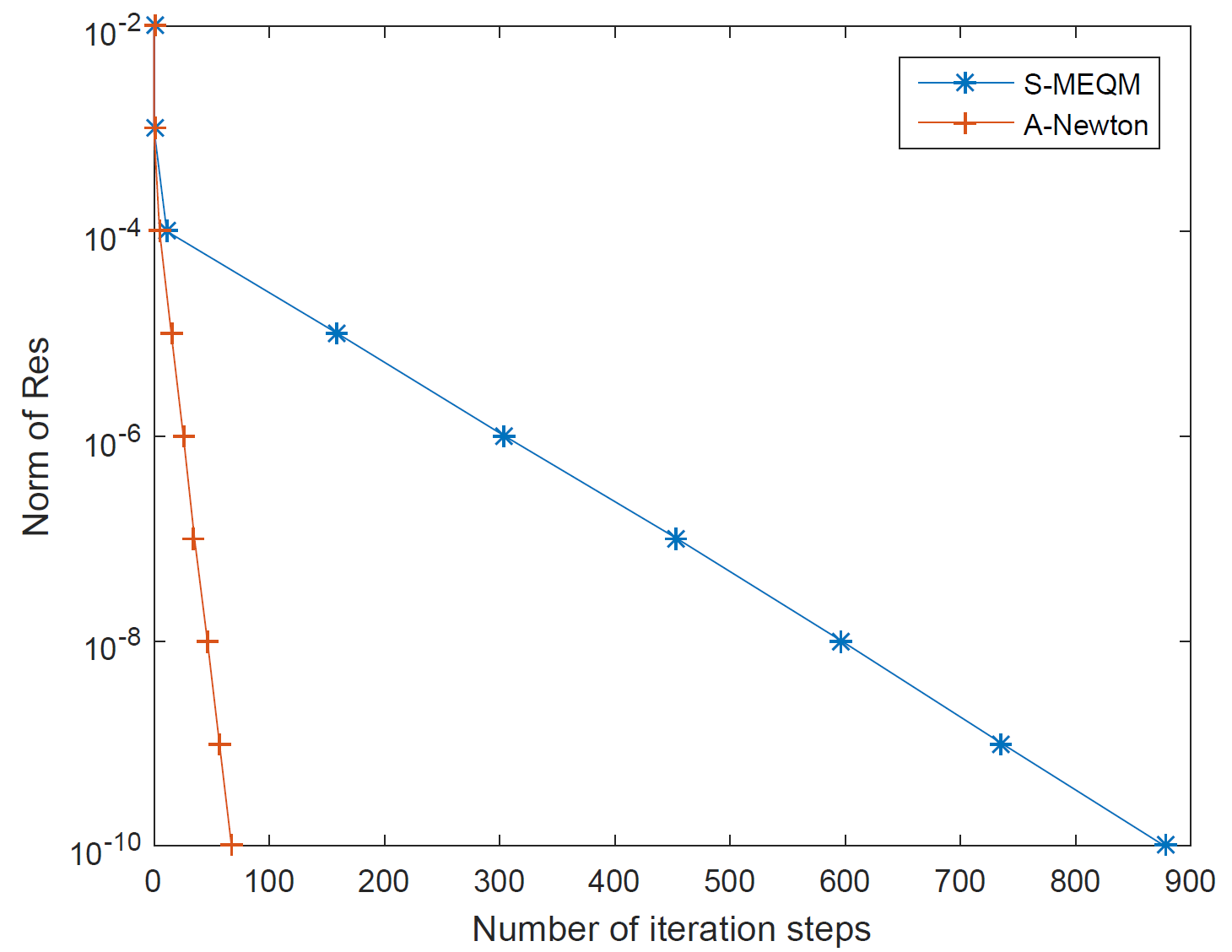}
\includegraphics[width=0.32\textwidth]{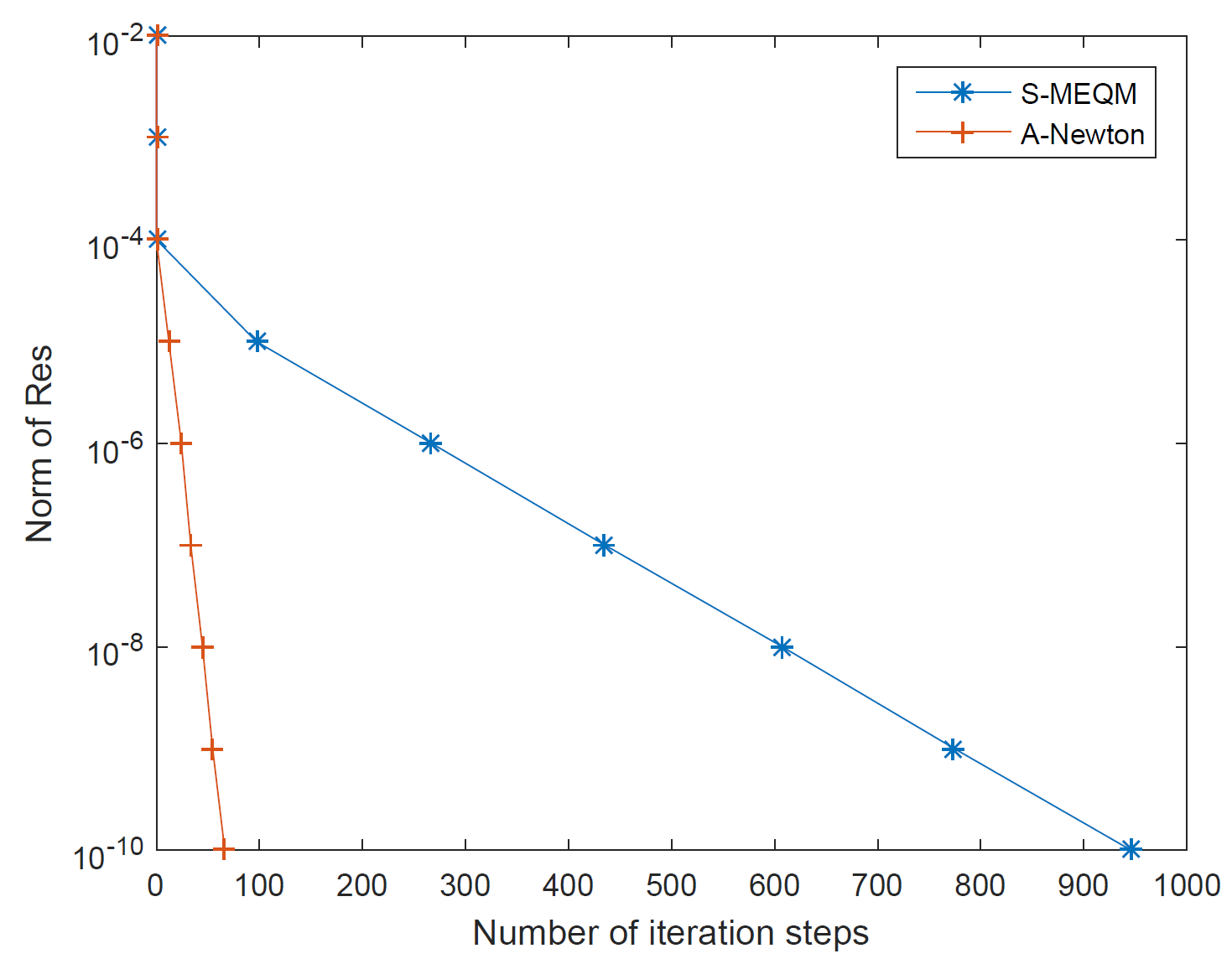}
\caption{\footnotesize Comparison between  S-MEQM and A-Newton on Problem 4: $n=10$ (left), $n=20$ (middle), $n=30$ (right)}

\end{figure}
\begin{figure}[H]
\centering
\includegraphics[width=0.34\textwidth]{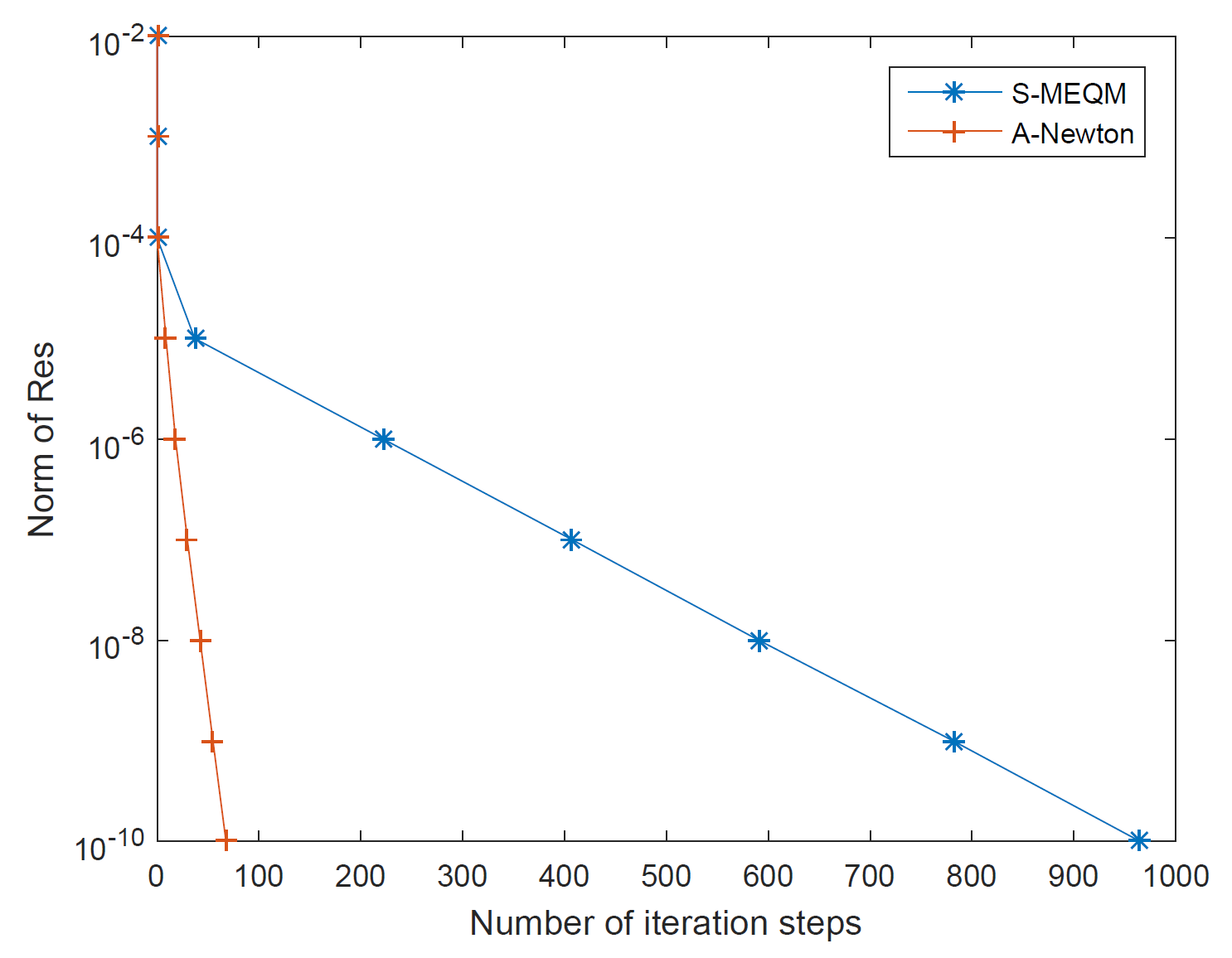}
\includegraphics[width=0.34\textwidth]{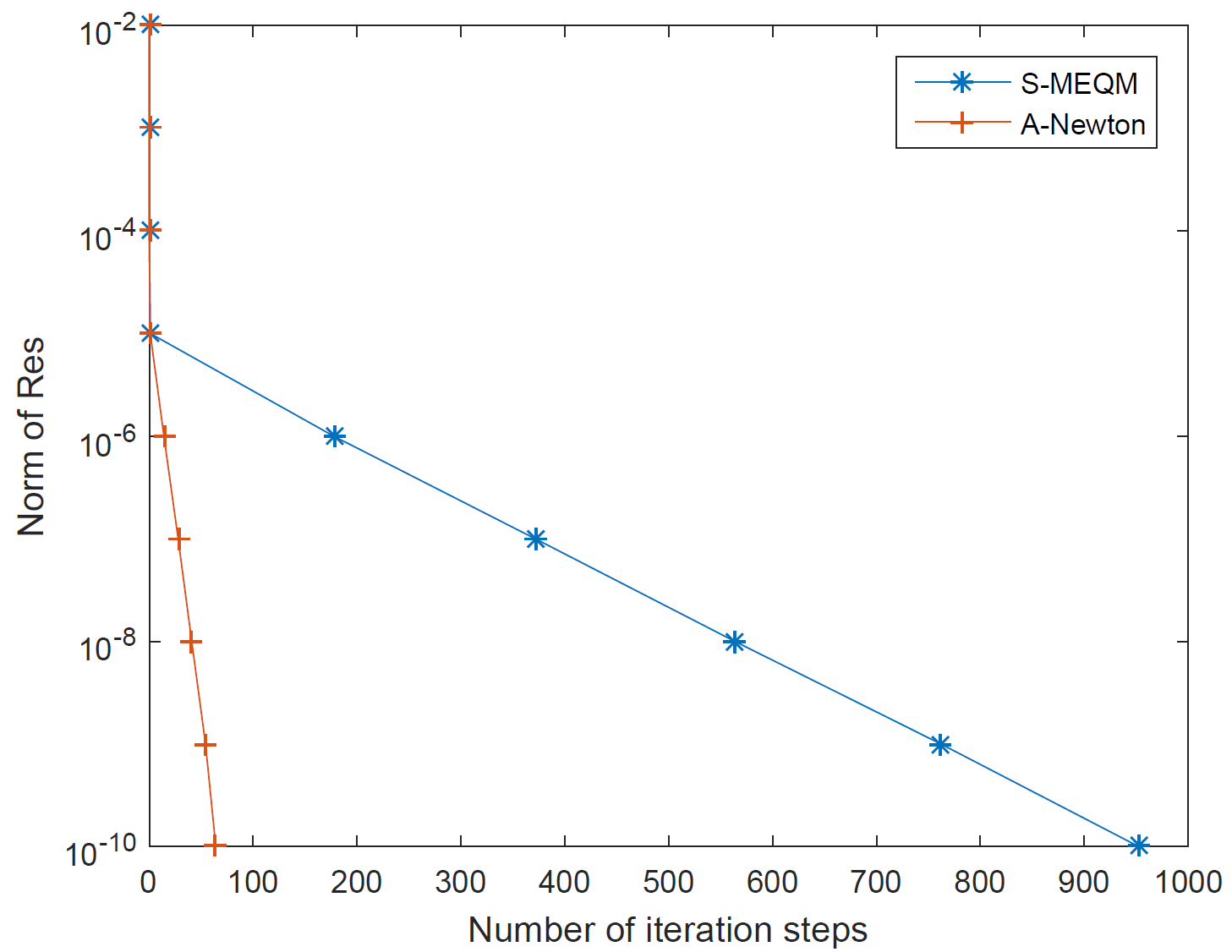}
\caption{\footnotesize Comparison between  S-MEQM and A-Newton on Problem 4: $n=40$ (left), $n=50$ (right)}
\end{figure}

\section{Conclusion}

We developed a sequential M-matrix equation based method and its improvement for solving M-tensor equations.
The methods can be regarded as approximate Newton methods. An advantage of the method is that
the subproblems of the methods are systems of linear equations with the same M-matrix as coefficient matrix.
If the initial point is appropriate chosen, the methods possess monotone convergence property.
Our numerical results show that when $\alpha \in (1,2)$ the performance of the sequence M-matrix method
can be better than the method with $\alpha\in (0,1]$. However, at the moment, we could not establish  the
related convergence theory. We leave it as a further research topic.

\vspace{3mm}

\noindent
{\bf Acknowledgement.} The authors would like to Prof. C. Ling of of Hangzhou Dianzi University who
constructed Example 2.1 which shows that an even order strong M-tensor equation may have multiple nonnegative solutions.

\end{document}